\documentclass[11pt]{article}

\usepackage[utf8]{inputenc}
\usepackage[T1]{fontenc}
\usepackage[a4paper,margin=1in]{geometry}
\usepackage{microtype}
\usepackage{mathtools}
\usepackage{cite}
\usepackage{amsmath,amssymb,amsfonts}
\usepackage{amsthm}
\usepackage{graphicx}
\usepackage{algorithm,algorithmic}
\usepackage{hyperref}
\hypersetup{hidelinks}
\usepackage{textcomp}

\newtheorem{rmk}{Remark}[section]
\newtheorem{proposition}{Proposition}[section]
\newtheorem{theorem}{Theorem}[section]
\newtheorem{lem}{Lemma}[section]   
\newtheorem{definition}{Definition}[section]

\def\BibTeX{{\rm B\kern-.05em{\sc i\kern-.025em b}\kern-.08em
T\kern-.1667em\lower.7ex\hbox{E}\kern-.125emX}}

\begin{document}

\title{Exponential stability of the {linearized} viscous Saint-Venant equations using a quadratic Lyapunov function}

\author{
Amaury Hayat\textsuperscript{1,2} and Nathan Lichtlé\textsuperscript{1,3}\\[0.5em]
\small \textsuperscript{1}Ecole Nationale des Ponts et Chaussées, Institut Polytechnique de Paris\\
\small \textsuperscript{2}Korean Institute for Advanced Study\\
\small \textsuperscript{3}University of California, Berkeley\\
\small \texttt{amaury.hayat@enpc.fr}, \texttt{nathan.lichtle@enpc.fr}
}
\date{}

\maketitle 

\begingroup
\renewcommand\thefootnote{}
\footnotetext{Published in \emph{IEEE Transactions on Automatic Control}; DOI: 10.1109/TAC.2025.3644127.}
\addtocounter{footnote}{-1}
\endgroup

\begin{abstract}

In this work, we investigate the exponential stability of the viscous Saint-Venant equations by adding to the standard hyperbolic Saint-Venant equations a viscosity term coming from the higher order approximation of the Saint-Venant equations from Navier-Stokes equations. The inclusion of viscosity transforms these equations into more complex second-order partial differential equations, accurately modeling the behavior of real-world fluids that inherently possess viscosity. We construct an explicit quadratic Lyapunov function and demonstrate that it must be diagonal in physical coordinates, revealing that certain quadratic Lyapunov functions effective in non-viscous cases become inadequate when viscosity is introduced. {We find explicit sufficient conditions on the parameters of the boundary conditions} such that for small viscosities a quadratic Lyapunov function exists. This result ensures the exponential stability of the linearized system around the steady-state solutions in the $L^2$ norm. 

\end{abstract}


\section{Introduction}

Introduced in 1871 by Barré de Saint-Venant, the Saint-Venant equations \cite{de1871theorie} have become widely used in the fields of hydraulics and fluid dynamics. These equations describe the one-dimensional evolution of a fluid, more specifically an incompressible flow, in a unidirectional open channel. They are significantly easier to solve than the two-dimensional shallow water equations, while remaining able to quite accurately model the dynamics of flows in open channels of arbitrary cross section, and have thus been playing a crucial role in the analysis of open-channel flows such as navigable rivers and canals in many practical applications, such as water level and flow rate control \cite{halleux2003boundary,gugat2003global,leugering2002modelling},{\cite{gugat2009global}}, \cite{gu2011exact,li2009exact,prieur2018boundary}. 

The Saint-Venant equations consist of two hyperbolic partial differential equations, describing the evolution of the fluid's depth and horizontal velocity across one-dimensional space and time, yielding a $2 \times 2$ nonlinear hyperbolic system. The goal is usually to prove the exponential stability of the steady states, meaning that deviations from an equilibrium state will diminish over time at an exponential rate. This is a highly valued property in practice because it ensures that the system quickly and predictably returns to a stable state after disturbances, which are ubiquitous in real-world applications.

Many tools have been developed over time for the problem of boundary feedback stabilization for such one-dimensional hyperbolic systems. \cite{GREENBERG198466} used the method of characteristics for nonlinear $2 \times 2$ homogeneous systems with $C^1$ solutions, later generalized to $n \times n$ homogeneous systems \cite{Qin}. For inhomogeneous hyperbolic systems, the backstepping method, introduced in \cite{doi:10.1137/1.9780898718607}, is a powerful tool for exponential stabilization \cite{KRSTIC2008750,6160338,Meglio2013StabilizationOA}. It has been used to prove the exponential stability of the linearized Saint-Venant-Exner equations with arbitrary slope or friction \cite{7402381}, as well as to stabilize a linearized bilayered Saint-Venant model \cite{DIAGNE2016130}. This method allows for exponential stabilization at an arbitrary rate, however it requires full-state feedback. In some cases, an observer can be designed to address this issue \cite{6160338,7402381,dimeglio:hal-01499940}. {\cite{arfaoui2011boundary} considers the stabilization of the viscous Saint-Venant equations using moving walls as a boundary control. More generally, \cite{bastin2023diffusion,bastin2025usefulness} study how viscosity influences the robustness of stability in hyperbolic systems with boundary feedbacks.}

In this work, we consider the method of quadratic Lyapunov functions, originally introduced in \cite{coron1999lyapunov}, and later extended to work for nonlinear hyperbolic systems of conservation laws \cite{coron2007strict,coron2008dissipative}. This method has the advantage of only requiring measurements at the boundaries, which is an important consideration for practical implementations. We build such Lyapunov functions as linear combinations of small perturbations around the steady states, which, coupled with appropriate boundary conditions, can be used to prove the exponential stability of the system around said steady states. \cite{bastin2017quadratic} uses such quadratic Lyapunov functions to show the exponential stability in $H^2$-norm of $2 \times 2$ hyperbolic systems such as Saint-Venant with nonuniform steady states. \cite{hayat2019quadratic} later applies a similar approach to a more general class of Saint-Venant equations with arbitrary friction and space-varying slope, then to generic density-velocity systems on a bounded domain \cite{hayat2021exponential}.{\cite{gugat2012h} constructs a Lyapunov function to establish stability in the $H^2$-norm for the isothermal Euler equations, which model gas flow in pipelines and share structural similarities with the viscous Saint-Venant equations considered here. In the case with friction, \cite{dick2011strict} derives a Lyapunov function to demonstrate stability in the $H^1$-norm.}{\cite{karafyllis2022spill,karafyllis2022feedback} address the problem of feedback stabilization for the transport of a tank containing liquid modeled by the viscous Saint-Venant equations, employing control Lyapunov functionals to design boundary feedback laws.}

We consider the Saint-Venant equations that have been augmented with a viscosity term $\mu$. Viscosity is a fundamental property of fluids that quantifies the internal friction force within them. In other words, viscosity measures the fluid's resistance to flow. Since the large majority of fluids we encounter in normal conditions have non-zero viscosity which will affect their flow behavior, taking viscosity into consideration is essential in understanding the dynamics of real-world flows. To this end, we consider the viscous Saint-Venant model derived in \cite{gerbeau2000derivation} from the Navier-Stokes system. \cite{mascia2005asymptotic} previously showed the asymptotic stability of the steady states of this model in the case where the equilibrium water speed is equal to 0. 

Our main objective in this work is to demonstrate the exponential stability of the linearized viscous Saint-Venant system, while providing constraints on the Lyapunov stability that characterize the system's behavior. Exponential stability is a desirable property as it guarantees the system's response will exponentially tend towards an equilibrium state, ensuring predictability and robustness in practical applications. Our approach involves linearizing the system around an equilibrium state, then constructing a quadratic Lyapunov function that provides valuable insights into the system's stability and convergence properties. Under a suitable choice of boundary conditions, we prove that the perturbations tend to zero exponentially fast in $L^2$ norm when the viscosity term is sufficiently small. Finally, we note that quadratic Lyapunov functions are often robust to non-linearities~\cite{bastin2016stability,hayat2023pi}, which is encouraging for the stability of the non-linear system.

The rest of the paper is organized as follows: in Section~\ref{sec:description}, we present the viscous Saint-Venant system that we consider, and perform a linearization to simplify our analysis. Section~\ref{sec:results} consists of our main results regarding the exponential stabi lity of the linearized viscous system, which are proved in Section~\ref{sec:exp_stab} along with our analytical findings. Finally, we conclude and discuss future work in Section~\ref{sec:conclusion}.

\section{Description of the viscous Saint-Venant system and the linearized system}
\label{sec:description} 

We consider the one-dimensional Saint-Venant model with viscosity, given by the following equations \cite{gerbeau2000derivation}:
\begin{align}
&H_t + (HV)_x = 0,
\label{eq:sv_nonlinear_1} \\
&\displaystyle V_t + VV_x + gH_x + \frac{f(H, V)}{H} - 4 \mu  \frac{(HV_x)_x}{H} = 0,
\label{eq:sv_nonlinear_2}
\end{align}
where $t \in [0, +\infty)$, $x \in [0,L]$, $H(t,x) : [0,+\infty) \times [0,L] \to (0,+\infty)$ is the water depth, $V(t,x) : [0,+\infty) \times [0,L] \to (0,+\infty)$ is the horizontal water speed and $g>0$ is the gravitational acceleration constant. 

In \eqref{eq:sv_nonlinear_2}, $\mu > 0$ is a viscosity coefficient and the term $f(H,V)$ represents a modified friction term, given by
\begin{equation}
\label{def:f}
    f(H,V) = \frac{\kappa V}{1 + \displaystyle \frac{\kappa H}{3\mu}},
\end{equation}
where $\kappa > 0$ denotes a small friction term. The derivation of the viscous Saint-Venant model from the Navier-Stokes equations can be found in \cite{gerbeau2000derivation}, where we divided the second equation by $H$ to arrive at (\ref{eq:sv_nonlinear_1},~\ref{eq:sv_nonlinear_2}).

{
\begin{rmk}
    The model \eqref{eq:sv_nonlinear_1}, \eqref{eq:sv_nonlinear_2}, derived from the Navier-Stokes equations 
    {in}~\cite{gerbeau2000derivation}, is the reference formulation for the viscous Saint-Venant equations. The derivation assumes a flat-bottom, rectangular channel with constant width and focuses on the laminar flow regime, resulting in a friction term that is linear in velocity. 
    Moreover, while the present work considers the flat-bottom case, the extension to sloped channels could still fall within the scope of our approach, especially when friction remains dominant. However, when the slope dominates, the situation is less clear, as previous work in the inviscid case relies on a Lyapunov function that does not directly generalize to the viscous case (see Remark~\ref{rmk:limitations}).
\end{rmk}
}

We consider an equilibrium $(H^*, V^*) : [0,L] \to (0, +\infty)^2$, that is, a time-invariant solution of the system defined by \eqref{eq:sv_nonlinear_1}, \eqref{eq:sv_nonlinear_2} and we 
assume that the system is in fluvial (i.e. subcritical) regime, that is 
\begin{equation}
    gH^* > V^{*2}.
    \label{eq:subcritical}
\end{equation}

\begin{rmk}[Subcritical flow] 
We consider the fluvial (i.e. subcritical) regime as it is often the most interesting and common case for navigable rivers, 
wherein waves can propagate both downstream and upstream~\cite{hayat2019boundary}. In the other case, the propagation speeds of the transport term have the same sign and the situation is easier to handle (see for instance \cite{hayat2021exponential}).
\end{rmk}

Our goal is to
find conditions of exponential stability of the linearized system around the steady state under some boundary conditions of the form :
\begin{equation}
\begin{split}
V(t,0) &= \mathcal B ( H(t,0), \mu, 0 ) \\
V(t,L) &= \mathcal B ( H(t,L), \mu, V_x(t,L) ) \\
V_x(t,0) &= 0
\end{split}
\label{eq:nonlinear_bc}
\end{equation} 
where $\mathcal B \in C^2( \mathbb R^3 ; \mathbb R)$ can correspond for instance to a control feedback law. {In particular, we are interested in seeing if a perturbed version of the most recent Lyapunov functions introduced in \cite{hayat2019quadratic} and \cite{bastin2017quadratic} can still work in this framework. Interestingly, we will show that perturbations of the Lyapunov function of \cite{hayat2019quadratic} do not work anymore, while the one found in \cite{bastin2017quadratic} can still be used.}

\begin{rmk}
The last condition in \eqref{eq:nonlinear_bc} implies in particular that the steady-state satisfies $V_{x}^{*}(0)=0$. This can be explained as specifying another value would destroy any hope of seeing $V^{*}$ converge (in $C^{1}([0,L])$-norm) when $\mu\rightarrow 0$, since the steady-states are constant functions when $\mu=0$.
\end{rmk}

We introduce the following proposition:

\begin{proposition}
\label{prop:estimVx}
For any $H_{0}>0$, $V_{0}>0$ satisfying
\begin{equation}
    gH_{0}-V_{0}^{2}>0,
\end{equation}
there exists a $\mu^{*}>0$ {(depending on $H_{0}$, $V_{0}$)} such that for any $\mu\in(0,\mu^{*})$, there is a unique steady-state $(H^{*},V^{*})\in C^{2}([0,L];(0,\infty))^{2}$  to \eqref{eq:sv_nonlinear_1}--\eqref{eq:sv_nonlinear_2} satisfying {\eqref{eq:subcritical} and} the last equation of \eqref{eq:nonlinear_bc} together with $H^{*}(0)=H_{0}$, $V^{*}(0)=V_{0}$. It satisfies $V^{*}_{x} \geq 0$ for all $x \in [0,L]$.
Moreover
\begin{equation}
\label{eq:estimVmu}
\begin{split}
& V^{*} = V_{0}+O(\mu),\\
    &V_{x}^{*} = \frac{V^{*}f(H^{*},V^{*})}{H^{*}(gH^{*}-{V^{*}}^{2})}-C_{0}\mu e^{-\frac{1}{4\mu}\int_{0}^{x} {\Gamma(s)} ds}+O(\mu^{2}),\\ &V_{xx}^{*} = \frac{C_{0}}{4V^{*}}\left(g H^{*}-{V^{*}}^{2}\right)e^{-\frac{1}{4\mu}\int_{0}^{x} {\Gamma(s)} ds}+ O(\mu),
    \end{split}
\end{equation}
 where $O(\mu)$ 
 (resp. $O(\mu^{2})$) refers to a function that tends to 0 at least as fast as $\mu$ (resp. $\mu^{2}$) in $L^{\infty}(0,L)$ when $\mu\rightarrow 0$, $C_{0}$ is a positive constant that only depend on $H_{0}$, $V_{0}$ and the parameters of the system and is bounded when $\mu\rightarrow 0$, and ${\Gamma}$ is a positive function that is bounded from above and below when $\mu\rightarrow 0$. Moreover $C_{0}$ is given by
 \begin{equation}
      C_{0} := \frac{1}{\mu}\left(\frac{V_{0}f(H_{0},V_{0})}{H_{0}(gH_{0}-V_{0}^{2})}\right),
 \end{equation}
 and ${\Gamma}$ is defined by
 \begin{equation}
    {\Gamma(x)} = \frac{gH^{*}(x)}{V^{*}(x)}-V^{*}(x) > 0.
 \end{equation}
 
\end{proposition}
\begin{rmk}
    Note that $f(H,V) / \mu$ is bounded when $\mu\rightarrow 0^{+}$ and hence $C_{0}$ is.
\end{rmk}
The existence of solutions and the estimates \eqref{eq:estimVmu} are not obvious a priori: indeed, the equations satisfied by the steady-states of \eqref{eq:sv_nonlinear_1}--\eqref{eq:sv_nonlinear_2} are singular ODEs (see \eqref{eq:sys1HV}) which can lead to very different behaviors when $\mu>0$ and $\mu=0$. {For example, $V_{xx}^{*}(0)$ converges to a non-vanishing constant when $\mu\rightarrow 0^{+}$, while $V_{x}^{*}=V_{xx}^{*}=0$ when $\mu=0$.} Proposition~\ref{prop:estimVx} is shown in Appendix \ref{app:estimVx}. 

We define the deviation of the state $(H,V)$ from the equilibrium $(H^{*},V^{*})$ as
\begin{equation}
    \begin{split}
        h(t,x) &\coloneqq H(t,x) - H^*(x), \\
        v(t,x) &\coloneqq V(t,x) - V^*(x).        
    \end{split}
\end{equation}
The linearized system around the steady states is then given by the following equation in matrix form:
\begin{equation}
A \begin{pmatrix}h \\ v\end{pmatrix}_{xx} + \begin{pmatrix}h \\ v\end{pmatrix}_t + B(x) \begin{pmatrix}h \\ v\end{pmatrix}_x + C(x) \begin{pmatrix}h \\ v\end{pmatrix} = 0,
\label{eq:linear_system}
\end{equation}
where
\begin{subequations}
\begin{equation}
A = \begin{pmatrix}0 & 0 \\ 0 & -4\mu \end{pmatrix},
\label{eq:linear_A}
\end{equation}
\begin{equation}
B = \begin{pmatrix}V^* & H^* \\ g - 4\mu \frac{V^{*}_x}{{H^{*}}} \;& \;{V^{*}} - 4\mu \frac{H^{*}_x}{{H^{*}}} \end{pmatrix},
\label{eq:linear_B}
\end{equation}
\begin{equation}
C = \begin{pmatrix} V^{*}_x & H^{*}_x \\ 
 {\frac{4\mu {H^{*}_x}{V^{*}_x} - f(H^{*}, V^{*})}{{H^{*}}^2} - \frac{f(H^{*}, V^{*})^{2}}{3 \mu {H^{*}} {V^{*}}}} & {{V^{*}_x} + \frac{f(H^{*}, V^{*})}{{H^{*}}  V^{*}}} \end{pmatrix}.
\label{eq:linear_C}
\end{equation}
\end{subequations}

{Note that $H^*, V^*, H^*_x, V^*_x$, and $V^*_{xx}$ are all functions of the space variable $x \in [0, L]$. For clarity of notation, we omit this dependence in the rest of the article.} The linearization of the boundary conditions \eqref{eq:nonlinear_bc} is given by:
\begin{equation}
\begin{split}
v(t,0) &= -b_0 h(t,0), \\
v(t,L) &= b_1 h(t,L) + \mu c_1 v_x(t,L), \\
v_x(t,0) &= 0,
\end{split}
\label{eq:linear_bc}
\end{equation}
where the coefficients are given as
\begin{equation}
\begin{split}
b_0 &= \frac{\partial \mathcal B}{\partial H}(H^*(0),\mu,0), \\
b_1 &= \frac{\partial \mathcal B}{\partial H}(H^*(L),\mu,V_x^*(L)), \\
\mu c_1 &= \frac{\partial \mathcal B}{\partial V_x}(H^*(L),\mu,V_x^*(L)).
\end{split}
\end{equation}

{We prove in Appendix~\ref{app:well_posedness_linear} that the linearized viscous system \eqref{eq:linear_system} with boundary conditions \eqref{eq:linear_bc} is well-posed.}

\section{Main results}
\label{sec:results}

We first recall the definition of the exponential stability:

\begin{definition}[Exponential stability]
The linearized system \eqref{eq:linear_system} with boundary conditions \eqref{eq:linear_bc} is exponentially stable 
if there exists $\gamma > 0$ and $C > 0$ such that for every initial condition $(h_0, v_0)^T \in L^2((0, L); \mathbb{R}^2)$, the Cauchy problem \eqref{eq:linear_system}, \eqref{eq:linear_bc} with initial condition $(h_0, v_0)$ has a unique solution $(h(t,x), v(t,x))\in C^{0}([0,+\infty); L^{2}(0,L))$, which satisfies exponential decay towards the origin starting from its initial condition:
\begin{equation}
\begin{split}
    \lVert (h(t, \cdot), v(t, \cdot))^T \rVert_{L^2((0, L) ; \mathbb R^2)} 
    \leq \;  C e^{-\gamma t} \lVert (h_0, v_0)^T \rVert_{L^2((0, L) ; \mathbb R^2)},
\end{split}
\end{equation}
for all $t \in [0, +\infty)$.
\end{definition}

There have been many results on the boundary stabilization of density-velocity systems such as Saint-Venant, most of them by finding appropriate 
basic quadratic Lyapunov functions, which for the $L^2$-norm can be defined as ${W}(y(t, \cdot)) = \int_0^L y^T(t,x) Q y(t,x) \, dx$ where $u=(h,v)$ is the state of the system.

In this work, we investigate the robustness of these results when a small viscosity term, likely present in real-world shallow water systems, is added. We {first} show that when there is viscosity, a basic quadratic Lyapunov function {for the linear system} must necessarily be diagonal in physical coordinates.

\begin{proposition}
Let ${W}$ be a functional on $L^{2}(0,L)$ defined by
\begin{equation}
    {W}(h,v) = \int_{0}^{L} \begin{pmatrix}h{(x)}\\ v{(x)}\end{pmatrix}^{T} Q{(x)} \begin{pmatrix}h{(x)}\\ v{(x)}\end{pmatrix} dx.
\end{equation}
If ${W}$ is a Lyapunov function for the system \eqref{eq:linear_system} in the $L^{2}$ norm, 
then $Q$ is diagonal.
\label{prop:Q_diag}
\end{proposition}
 
{We prove this result in Appendix~\ref{app:proof_Q_diag}.}
A similar result is known in Riemann coordinates: after a change of variable to make the term differentiated with respect to $x$ diagonal, the Lyapunov function must be diagonal \cite{bastin2011boundary}. Here the constraint we show is even stronger, the Lyapunov function needing to be diagonal {in Riemann coordinates as well as} in physical coordinates $(h \; v)^T$. In particular, this implies that the Lyapunov functions found in \cite{hayat2019quadratic} and \cite{hayat2021exponential} are not suitable anymore when a small viscosity term is added.

{ For instance, the Lyapunov function in \cite[Proposition 4]{hayat2019quadratic} has the form
{\small \[ W(h,v) = \int_0^L \frac{q_1 + q_2}{H^*} \left( gh^2 + 2 \frac{q_1 - q_2}{q_1 + q_2} \sqrt{gH^*}hv + H^*v^2 \right) dx \]}
and does not generalize to the viscous case (in the sense of Proposition~\ref{prop:Q_diag}) due to the presence of the cross-term $hv$.  
}

If the Lyapunov is diagonal in physical coordinates, as is the case in \cite{bastin2017quadratic} for instance, then it can still be used in the viscous case. In this case, we can deduce stability conditions that converge back to those in \cite{bastin2017quadratic} when the viscosity $\mu$ tends to 0. 

\begin{theorem}
\label{th:main}
{Let $(H^{*},V^{*})$ be a steady state of \eqref{eq:sv_nonlinear_1}, \eqref{eq:sv_nonlinear_2}, as constructed in Proposition~\ref{prop:estimVx} for some $(H_0,V_0)$ satisfying the conditions therein. Suppose further that $(H^{*},V^{*})$ satisfies the} boundary conditions \eqref{eq:nonlinear_bc}, the subcritical condition \eqref{eq:subcritical} and 
\begin{equation}
\label{eq:assumpsteady}
gH^{*}(0)<(2+\sqrt{2})V^{*2}(0).
\end{equation}
Assume the boundary conditions coefficients in \eqref{eq:linear_bc} satisfy 
\begin{equation}
    b_0 \in (b_0^-, b_0^+), \quad b_1 \in \mathbb{R} \setminus [b_1^-, b_1^+], \quad c_1 \in (c_1^-, c_1^+),
    \label{eq:bc_coef_constraints}
\end{equation}
for $b_0^\pm$, $b_1^\pm$ and $c_1^\pm$ {defined as:}
{
\begin{align}
    b_0^\pm &= g \left( \frac{1}{V^*(0)} \pm \sqrt{ \frac{1}{V^*(0)^2} - \frac{1}{g H^*(0)} } \right), \label{eq:bc_b0} \\
    b_1^\pm &= g \left( - \frac{1}{V^*(L)} \pm \sqrt{ \frac{1}{V^*(L)^2} - \frac{1}{g H^*(L)} } \right) \label{eq:bc_b1}, \\
    c_1^\pm &= \frac{4 \left( - V^* - b_1 H^* \pm \sqrt{V^* (H^*V^*b_1^2/g + 2 H^*b_1 + V^*)} \right)}{H^* g - V^{*2}}. \label{eq:bc_c1}
\end{align}
}
Then there exists ${\mu_{\circ}}>0$ such that for any $\mu\in(0,{\mu_{\circ}})$ {the linearized system \eqref{eq:linear_system} with boundary conditions \eqref{eq:linear_bc} 
is exponentially stable for the ${L^{2}}$-norm}.
\end{theorem}
{
\begin{rmk}[Subcritical condition]
Note that we only assume the steady state $(H^*, V^*)$ of the nonlinear system is subcritical; no such assumption is made for the solution $(h, v)$ of the linearized system. For the linear system, the propagation speeds only depend on $(H^*, V^*)$, so the behavior of the regime only depends on $(H^*, V^*)$. To extend the result to local exponential stability of the nonlinear system, one should then assume $h$ and $v$ are sufficiently small so that the full solution $(h+H^{*},v+V^{*})$ remains in the subcritical regime.
\end{rmk}
}
\begin{rmk}[{Limitations}] \label{rmk:limitations}
Compared to the results of {\cite{bastin2017quadratic}}
and \cite{hayat2019quadratic}, not all steady-states can be stabilized but only those satisfying 
\eqref{eq:assumpsteady}. Concerning \cite{hayat2019quadratic} this is due to the fact that, as surprising as it may seem, the optimal Lyapunov function for the inviscid system used in \cite{hayat2019quadratic} and \cite{hayat2021exponential} cannot be used anymore when viscosity appears. Concerning \cite{bastin2017quadratic}, the difference comes from the form of the source term $f(H^{*},V^{*})/H^{*} \sim  V^{*}/{H^{*}}^{2}$ while in the Saint-Venant equations considered in \cite{bastin2017quadratic} the source term is of the form ${V^{*}}^{2}/H^{*}$. With a source term of the form ${V^{*}/{H^{*}}^{2}}$, the same 
limitation would occur using the results of \cite{bastin2017quadratic}.
\end{rmk}

\section{Exponential stability of the linearized system}
\label{sec:exp_stab}

In this Section, we show Theorem~\ref{th:main}. To do so, let us consider a steady state $(H^*,V^*)$ of \eqref{eq:sv_nonlinear_1}, \eqref{eq:sv_nonlinear_2} that satisfies \eqref{eq:nonlinear_bc}, \eqref{eq:subcritical} and \eqref{eq:assumpsteady}. We introduce the following quadratic Lyapunov function
\begin{equation}
    {\bf {W}}(y) = \int_0^L y^T{(t,x)}Q(x)y{(t,x)} \; dx,
    \label{eq:lyap_fn}
\end{equation}
where 
\begin{equation}
\label{eq:defQ}
    y := \begin{pmatrix} h \\ v \end{pmatrix},\;\; Q =\text{diag}(q_{1},q_{2}):= \begin{pmatrix} g + \mu \widetilde{q_1} & 0 \\ 0 &H^{*} + \mu \widetilde{q_2}\end{pmatrix}.
\end{equation}
{Here} $\widetilde{q_1}$ and $\widetilde{q_2}$ are chosen as
\begin{equation}
\label{eq:defq1q20}
\widetilde{q}_{2} = H^{*},\; \widetilde{q_{1}} = g -4 (1+\mu)\frac{V_{x}^{*}}{H^{*}},
\end{equation}
such that 
\begin{equation}
\label{eq:defq1q2}
H^{*} (g+\mu\widetilde{q_1}) = (H^{*}+\mu\tilde{q}_{2})\left(g - 4 \mu \frac{V^{*}_x}{H^{*}}\right),
\end{equation}
which is equivalent to saying that $QB$ is symmetric.
Note that $QA$ is also symmetric since $Q$ and $A$ are both diagonal and that when $\mu \rightarrow 0$ we recover the Lyapunov function introduced {in \cite[Equation (17)]{bastin2017quadratic}, which is defined as $q_1 = g$ and $q_2 = H^*$.}

{
In order to show Theorem \ref{th:main}, it is enough to show that the Lyapunov function satisfies the exponential stability properties below.

\begin{proposition}
    The Lyapunov function \eqref{eq:lyap_fn} satisfies the following properties:
    \begin{enumerate}
        \item ${\bf {W}}(y)$ 
        is equivalent to the square of the $L^{2}$ norm meaning that there exists positive constants $c_{1}$, $c_{2}$ such that for any $y\in L^{2}((0,L);\mathbb{R}^{2})$
        \[ c_1 \| y \|_{L^2}^{{2}} \leq {\bf {W}}(y) \leq c_2 \| y \|_{L^2}^{{2}}.\] Note that this implies positive definiteness (${\bf {W}}(y) > 0, \; \forall y \neq (0, 0)^T$, and ${\bf {W}}((0, 0)^T) = 0$) as well as the radially unbounded property (${\bf {W}}(y) \to +\infty$ as $\| y \|_{L^2} \to +\infty$).
        \item If \eqref{eq:bc_coef_constraints} is satisfied, then there exists $\gamma>0$ such that for any 
        $y\in C^{1}([0,{+\infty)};L^{2}(0,L))$ solution to \eqref{eq:linear_system} with boundary conditions \eqref{eq:linear_bc}, \[\frac{d}{dt}({\bf {W}}(y(t,\cdot))) \leq - \gamma {\bf {W}}(y(t,\cdot)).\]
    \end{enumerate}
    \label{prop:lyap_3_props}
\end{proposition}
{Indeed, assume that Proposition \ref{prop:lyap_3_props} holds. Then from Appendix \ref{app:well_posedness_linear}, for any initial condition $(h_{0},v_{0})^{T}$ the system \eqref{eq:linear_system}, \eqref{eq:linear_bc} has a unique solution $y=(h,v)\in C^{0}([0,+\infty);L^{2}(0,L))$ and from Proposition \ref{prop:lyap_3_props},
\begin{equation}
    c_{1}\|(h,v)(t,\cdot)\|_{L^{2}}^{2}\leq {\bf {W}}(y(t,\cdot)) \leq e^{-\gamma t}{\bf {W}}(y(0,\cdot))\leq c_{2} \|(h_{0},v_{0})\|_{L^{2}}^{2},
\end{equation}
and since $\gamma$, $c_{1}$ and $c_{2}$ do not depend on $y$ this ends the proof of Theorem \ref{th:main}.
}

{Let us prove Proposition \ref{prop:lyap_3_props}.} Property 1) directly follows from the fact that $Q$ is positive definite for $\mu \in (0, \mu_Q)$, for some $\mu_Q > 0$ that can be computed explicitly. The rest of this section proves property 2).
}

We now compute the time derivative of ${\bf {W}}(y(t,\cdot))$ along the solutions of the system. Using \eqref{eq:linear_system} and the symmetry of $Q$, we get:
\begin{align}
    \dot{\bf {W}} &= \int_0^L 2 y^T Q y_t \; dx \\
    &= -2 \int_0^L y^T Q (Ay_{xx} + By_x + Cy) dx
\end{align}
Simplifying further, we obtain
\begin{equation}
\label{eq:afterint}
\dot{\bf {W}} + \gamma {\bf {W}} = \mathcal I + \mathcal B
\end{equation}
where the integral and boundary terms are given as follows:
\begin{align}
\begin{split}
    \mathcal I &= \int_0^L y^T\left(\gamma Q - (QC + (QC)^{T}) - Q_{xx}A + (QB)_x\right)y \; dx  + 2\int_0^L y^T_x QA y_x \; dx,
    \label{eq:lyap_I}
\end{split} \\
\begin{split}
    \mathcal B &= \left[ y^T(Q_xA - QB)y -2 y^TQAy_x \right]_0^L.
    \label{eq:lyap_B}
\end{split}
\end{align}
The full derivation can be found in Appendix~\ref{app:lyapunov_derivation}. 
{We have the two following claims about $\mathcal{I}$ and $\mathcal{B}$ respectively:} 

\begin{proposition}
\label{prop:inner}
{Let $\phi(\mu{, \gamma}) =  \gamma Q - (QC+(QC)^{T}) - Q_{xx}A + (QB)_x$, defined for $\mu > 0$.} There exists $\mu_{{a}}>0$ such that for any viscosity $\mu\in(0,\mu_{{a}})$, there exists $\gamma>0$ such that $\phi(\mu{, \gamma})$ is negative definite for any $x\in[0,L]$.     
\end{proposition}
{This is shown in Section \ref{ssec:int}.}
\begin{proposition}
\label{prop:boundary}
Assume that $b_0, b_1$ and $c_1$ satisfy \eqref{eq:bc_coef_constraints}. There exists ${\mu_{b}}>0$ such that for any viscosity $\mu\in(0,{\mu_{b}})$, 
$$\mathcal{B} \leq 0.$$
\end{proposition}
{This is shown in Section \ref{ssec:bound}.
Finally, from \eqref{eq:afterint} and \eqref{eq:lyap_I}, for any $\mu\in(0,\min(\mu_{a},\mu_{b}))$ property 2) holds and therefore, choosing ${\mu_{\circ}} = \min(\mu_{a},\mu_{b},\mu_{Q})$, Proposition \ref{prop:lyap_3_props} holds. 
{The two following subsections are devoted to the proof of Proposition \ref{prop:inner} and Proposition \ref{prop:boundary}.}
}

\subsection{Integral term}
\label{ssec:int}
\begin{proof}[Proof of Proposition \ref{prop:inner}]
We denote $D := - (QC+(QC)^{T}) - Q_{xx}A + (QB)_x$. 
{
Observe that from \eqref{eq:defQ}, \eqref{eq:defq1q2} and \eqref{eq:defq1q20}
\begin{gather}
(Q B) = \begin{pmatrix}
 (g+\mu \tilde{q}_{1}) V^{*} & (g+\mu \tilde{q}_{1}) H^{*}\\
 \ast 
 \;\; &\;\; (H^{*}+\mu \tilde{q}_{2})(V^{*}-4\mu \frac{H^{*}_{x}}{H^{*}})
\end{pmatrix}\\
    (QB)_{x} = \begin{pmatrix}
        g{V^{*}_{x}} +\mu (\tilde{q}_{1}{V^{*}})_x &  (g+\mu \tilde{q}_{1}){H^{*}_{x}}+\mu\tilde{q}_{1x}{H^{*}}\\
        \ast & -4\mu {H^{*}_{xx}}-4\mu^{2}{{H^{*}_{xx}}}
    \end{pmatrix}
\end{gather}
{where we used the fact that $(H^{*}V^{*})_x = 0$, from the steady state of \eqref{eq:sv_nonlinear_1}.
}
Denoting $f_{1}(H,V) = -f(H,V)(3\mu+2\kappa H)/(3\mu H+\kappa H^{2})$, we obtain
 \begin{equation}
  (QC+(QC)^{T}) = \begin{pmatrix}
      2(g+\mu \tilde{q}_{1}){V^{*}_{x}} & \omega_1 \\
      \ast & \omega_2
  \end{pmatrix}  
\end{equation} 
with
\begin{align}
\begin{split}
    \omega_1 &= (g+\mu \tilde{q}_{1}){H^{*}_{x}}+({H^{*}}+\mu \tilde{q}_{2}) \left( {\frac{f_{1}({H^{*}},{V^{*}})}{H^*}}+4\mu \frac{{H^{*}_{x}}{V^{*}_{x}}}{{H^{*}}^{2}}\right),
\end{split}\\
    \omega_2 &= 2({H^{*}}+\mu \tilde{q}_{2})\left({V^{*}_{x}}+\frac{f({H^{*}},{V^{*}})}{{H^*}{V^{*}}}\right).
\end{align}
Thus
 \begin{equation}
\begin{split}
\label{eq:expressionD}
    D &= { - } (QC+(QC)^{T}) - Q_{xx}A + (QB)_x\\
    &=\begin{pmatrix}
          -(g+\mu \tilde{q}_{1}){V^{*}_{x}} +\mu\tilde{q}_{1x}{V^{*}} & \omega_3 \\
      \ast & \omega_4
    \end{pmatrix},
    \end{split}
\end{equation}
with
\begin{align}
    \omega_3 &=  \mu\tilde{q}_{1x}{H^{*}}
          -({H^{*}}+\mu \tilde{q}_{2})\left(\frac{f_{1}({H^{*}},{V^{*}})}{{H^{*}}}+4\mu \frac{{H^{*}_{x}}{V^{*}_{x}}}{{H^{*}}^{2}}\right), \\
    \omega_4 &= -{ 2} ({H^{*}}+\mu \tilde{q}_{2})\left({V^{*}_{x}}+\frac{f({H^{*}},{V^{*}})}{H^{*}V^{*}}\right).
\end{align}
}

Let us define 
\begin{equation}
    f_{1}(H,V) := -f(H,V)(3\mu+2\kappa H)/(3\mu H+\kappa H^{2})
\end{equation} and following the expression of $D$ given by \eqref{eq:expressionD} 
we obtain
\begin{equation}
\label{eq:estimdetD0}
\begin{split}
&det(D) =  \left((g+\mu \tilde{q}_{1}){V^{*}_{x}} -\mu\tilde{q}_{1x}{V^{*}} \right) \left[
{2}({H^{*}}+\mu \tilde{q}_{2})({V^{*}_{x}}+\frac{f({H^{*}},{V^{*}})}{{{H^{*}}V^{*}}})
\vphantom{\frac{f}{4}}\right]   \\
& \quad\quad\quad\quad-\left[\mu\tilde{q}_{1x}{H^{*}}
          -({H^{*}}+\mu \tilde{q}_{2})\left(\frac{f_{1}({H^{*}},{V^{*}})}{{H^{*}}}\right.\right. \left.\left.+4\mu \frac{{H^{*}_{x}}{V^{*}_{x}}}{{H^{*}}^{2}}\right)\right]^{2}.
\end{split}
\end{equation}

From Proposition \ref{prop:estimVx}, and since $f(H^{*}, V^{*}) = O(\mu)$, we have $V^{*}_{x} = O(\mu)$ and we can get the following estimates for $\tilde{q}_{1}$ and $\tilde{q}_{2}$:
\begin{equation}
\label{eq:estimq1q2}
\begin{split}
&\tilde{q}_{1} = g+O(\mu),\\
&\tilde{q}_{1x} = {-\frac{C_{0}}{H^{*}V^{*}}\left(g H^{*}-{V^{*}}^{2}\right)e^{-\frac{1}{4\mu}\int_{0}^{x}{\Gamma(s)}ds}} +O(\mu),\\
&\tilde{q}_{2xx} = O(1),
\end{split}
\end{equation}
 where $O(1)$ (resp. $O(\mu)$) refers to a function that is bounded (resp. tends to 0 at least as fast as $\mu$) in $L^{\infty}(0,L)$.

From Proposition \ref{prop:estimVx} and \eqref{eq:estimq1q2}, \eqref{eq:estimdetD0} becomes
\begin{equation}
    \label{eq:estimdetD1}
\begin{split}
det(D) =  {2}g{V^{*}_{x}}^{2}{H^{*}}
+{2}g{V^{*}_{x}}\frac{f({H^{*}},{V^{*}})}{{V^{*}}} -f_{1}({H^{*}},{V^{*}})^{2}+{\mathcal{Q}}+O(\mu^{3}),
\end{split}
\end{equation}
{where
\begin{equation}
    \begin{split}
        \mathcal{Q}:=&-2\mu \widetilde{q}_{1x}V^{*}H^{*}\left(V_{x}^{*}+\frac{f(H^{*},V^{*})}{H^{*}V^{*}}\right) -\mu^{2} \widetilde{q}_{1x}^{2}{H^{*}}^{2}+2\mu \widetilde{q}_{1x}H^{*}f_{1}(H^{*},V^{*})
    \end{split}
\end{equation}
}

{ Let us study $\mathcal{Q}$:
\small \begin{equation}
\begin{split}
\label{eq:estimQ}
\mathcal{Q}
=&-\mu\tilde{q}_{1x}\left[2V^{*}H^{*}\left(V_{x}^{*}+\frac{f(H^{*},V^{*})V^{*}}{H^{*}(gH^{*}-{V^{*}}^{2})}\left(\frac{g H^{*}}{{V^{*}}^{2}}-1\right)\right)-2f_{1}(H^{*},V^{*})H^{*}+{H^{*}}^{2}\mu\tilde{q}_{1x}\right]\\
=&-\mu\tilde{q}_{1x}\left[2V^{*}H^{*}\left(V_{x}^{*}+\frac{f(H^{*},V^{*})V^{*}}{H^{*}(gH^{*}-{V^{*}}^{2})}\left(\frac{g H^{*}}{{V^{*}}^{2}}-1\right)-\frac{C_{0}\mu}{2}\left(\frac{g H^{*}}{{V^{*}}^{2}}-1\right)e^{-\frac{1}{4\mu}\int_{0}^{x}{\Gamma(s)}ds}
\right)\right.\\
&\left.-2f_{1}(H^{*},V^{*})H^{*}+O(\mu^{2})\right]\\
=&-\mu\tilde{q}_{1x}\left[2V^{*}H^{*}\left(\frac{gH^{*}}{{V^{*}}^{2}}V_{x}^{*}+\frac{C_{0}\mu}{2}\left(\frac{g H^{*}}{{V^{*}}^{2}}-1\right) 
\left( e^{-\frac{1}{4\mu}\int_{0}^{x}{\Gamma(s)}ds}\right)\right) -2f_{1}(H^{*},V^{*})H^{*}+O(\mu^{2})\right]  \\
=&\mu^{2}\frac{C_{0}}{H^{*}V^{*}}(gH^{*}-{V^{*}}^{2})e^{-\frac{1}{4\mu}\int_{0}^{x}{\Gamma(s)}ds}\Big[2V^{*}H^{*}\left( \mu^{-1} \frac{gH^{*}}{{V^{*}}^{2}}V_{x}^{*}+\frac{C_{0}}{2}\left(\frac{g H^{*}}{{V^{*}}^{2}}-1\right)e^{-\frac{1}{4\mu}\int_{0}^{x}{\Gamma(s)}ds}\right)\\
&-2\mu^{-1}f_{1}(H^{*},V^{*})H^{*}\Big]+O(\mu^{3}).
\end{split}
\end{equation}
\normalsize
Note that, from the definition of $f$ and $f_{1}$, and since $H^{*}\geq 0$, $C_{0}\geq 0$ and $V^{*}\geq 0$, $f_{1}(H^{*},V^{*})\leq 0$. Also, recall that $C_{0}$ is a positive constant and $gH^{*}>{V^{*}}^{2}$ from \eqref{eq:subcritical}. Finally, recall that from Proposition \ref{prop:estimVx} we have that for any $\mu\in(0,\mu^{*})$, $V_{x}^{*} \geq 0$. Using all this in \eqref{eq:estimQ}
\begin{equation}
    \mathcal{Q} = \mathcal{Q}^{0}+O(\mu^{3}),
\end{equation}
where $\mathcal{Q}^{0}\geq 0$. Using this in \eqref{eq:estimdetD1}, as well as the definition of $f_{1}({H^{*}},{V^{*}})$, we have
}
\begin{equation}
    \begin{split}
        det(D) &= \left(2g{V_{x}^{*}}H^{*}+2g\frac{f(H^{*},V^{*})}{V^{*}}\right)V_{x}^{*} -\frac{4(f(H^{*},V^{*}))^{2}}{{H^{*}}^{2}}
        +\mathcal{Q}_{0}+O(\mu^3).
    \end{split}
\end{equation}
From Proposition~\ref{prop:estimVx}, we have
\begin{equation}
  \mu C_{0} e^{-\frac{1}{4\mu}\int_{0}^{x}{\Gamma(s)}ds} = S +O(\mu^{2}) \geq 0
\end{equation}
with $C_0$ and ${\Gamma(s)}$ defined in Proposition~\ref{prop:estimVx} and
\begin{equation}
    S \coloneqq \frac{V^{*}f(H^{*},V^{*})}{H^{*}(gH^{*}-{V^{*}}^{2})} - V_{x}^{*}
\end{equation}
Using this in \eqref{eq:estimQ} (and the definition of $f_{1}$), we have
\begin{equation}
\label{eq:estimint0}
\begin{split}
    \mathcal{Q}^{0} &= S\frac{gH^{*}-{V^{*}}^{2}}{H^{*}V^{*}}
    \left[
    2g{H^{*}}^{2} \frac{V_{x}^{*}}{V^{*}}+ SV^{*}H^{*}\left(\frac{gH^{*}}{{V^{*}}^{2}}-1\right)+4 f(H^{*},V^{*})
    \right]+O(\mu^{3})
\end{split}
\end{equation}
{(note that $f(H^*,V^*)$, $V_x^*$ and thus $S$ are all $O(\mu)$)}. As a consequence, developping $S$ in the last term of \eqref{eq:estimint0} we have
\begin{equation}
\label{eq:defD0}
    \begin{split}
           &det(D) = \left(2g{V_{x}^{*}}H^{*}+2g\frac{f(H^{*},V^{*})}{V^{*}}\right)V_{x}^{*}-\frac{4(f(H^{*},V^{*}))^{2}}{{H^{*}}^{2}}\\
        &+S\frac{gH^{*}-{V^{*}}^{2}}{H^{*}V^{*}}
    \left[
    2g{H^{*}}^{2} \frac{V_{x}^{*}}{V^{*}}+ SV^{*}H^{*}\left(\frac{gH^{*}}{{V^{*}}^{2}}-1\right)\right]\\
    &+4 \left(\frac{f(H^{*},V^{*})}{H^{*}}\right)^{2}-4f(H^{*},V^{*}) V_{x}^{*}\frac{gH^{*}-{V^{*}}^{2}}{H^{*}V^{*}}
    +O(\mu^3)\\
    &= D_{0}(\mu,x)+O(\mu^3)
    \end{split}
\end{equation}
where
\small
\begin{equation}
    \begin{split}
    &D_0(\mu, x) \coloneqq S\frac{gH^{*}-{V^{*}}^{2}}{H^{*}V^{*}}
    \left[
    2g{H^{*}}^{2} \frac{V_{x}^{*}}{V^{*}}+ SV^{*}H^{*}\left(\frac{gH^{*}}{{V^{*}}^{2}}-1\right)\right] \\
    &+\left(2g{V_{x}^{*}}H^{*}+2g\frac{f(H^{*},V^{*})}{V^{*}}
    -4f(H^{*},V^{*}) \frac{gH^{*}-{V^{*}}^{2}}{H^{*}V^{*}}
    \right)V_{x}^{*}
    \end{split}
\end{equation}
\normalsize
Note that from Proposition \ref{prop:estimVx}, $V_{x}^{*}\geq 0$ and $S\geq 0$.
Also from Proposition \ref{prop:estimVx}, we have
\begin{equation}
    \begin{split}
    &\lim\limits_{\mu\rightarrow 0}\mu^{-1}S\rightarrow 0,\;\;\text{ if }x\in(0,L],\\
    &V_{x}^*(0) = 0,
    \end{split}
\end{equation}
and $f({H^{*}},{V^{*}})\mu^{-1}$ converges (in $C^{0}([0,L]$) to a positive function $\tilde{f}$ from \eqref{def:f}. As a consequence
\begin{equation}
\label{eq:lim0}
    \lim\limits_{\mu\rightarrow 0}\mu^{-2} D_{0}(\mu,0)= \left(\frac{V_{0}\tilde{f}(0)}{H_{0}}\right)^{2}H_{0}\left(\frac{gH_{0}}{{V_{0}}^{2}}-1\right),
\end{equation}
and, for $x\in(0,L]$
\begin{equation}
\label{eq:limn0}
\begin{split}
    \lim\limits_{\mu\rightarrow 0}\mu^{-2} &D_{0}(\mu,x)= \frac{V_{0}\tilde{f}^{2}(x)}{H_{0}(gH_{0}-V_{0}^{2})}
    \left[
    2g\frac{V_{0}}{gH_{0}-V^{2}_{0}}+\frac{2g}{V_{0}}-4\frac{gH_{0}-V_{0}^{2}}{H_{0}V_{0}}
    \right].
    \end{split}
\end{equation}
{Note that we used the fact that $V^{*}(x)\rightarrow V_{0}$ and $H^{*}(x)\rightarrow H_{0}$ when $\mu\rightarrow 0$ for any $x>0$ (see Proposition \ref{prop:estimVx}).}
{
We have
\begin{equation}
\begin{split}
        &2g\frac{V_{0}}{gH_{0}-V^{2}_{0}}+\frac{2g}{V_{0}}-4\frac{gH_{0}-V_{0}^{2}}{H_{0}V_{0}}\\
        =&\frac{1}{H_{0}V_{0}(gH_{0}-V^{2}_{0})}\left(2(gH_{0})^{2}-4(gH_{0}-V_{0}^{2})^{2}\right)\\
        =&\frac{1}{H_{0}V_{0}(gH_{0}-V^{2}_{0})}\left(-2(gH_{0})^{2}+8gH_{0}V_{0}^{2}-4 V_{0}^{4}\right).
        \end{split}
\end{equation}
A simple polynomial analysis shows that $-2(gH_{0})^{2}+8gH_{0}V_{0}^{2}-4V_{0}^{4}>0$ as long as $(2+\sqrt{2})V_{0}^{2}>gH_{0}>(2-\sqrt{2})V_{0}^{2}$, which is true here from \eqref{eq:assumpsteady} and \eqref{eq:subcritical}. This together with \eqref{eq:lim0} and \eqref{eq:limn0} implies that there exists $c>0$ such that
\begin{equation}
\lim\limits_{\mu\rightarrow 0}\min(\mu^{-2}D_{0}(\mu,x),c) = c>0,\;\text{ for any }x\in[0,L].
\end{equation}
Note that $(\mu,x)\mapsto \min(\mu^{-2}D_{0}(\mu,x),c)$ is continuous on $(0,\mu^{*}]\times [0,L]$,  where $\mu^{*}$ is given by Proposition \ref{prop:estimVx}. From the expression of $D_{0}$ and Proposition \ref{prop:estimVx}, there exists $C>0$ independent of $\mu$ (but maybe depending on $\mu^{*}$) such that $|\partial_{x}D_{0}(\mu,\cdot)|_{L^{\infty}}<C$ for any $\mu\in(0,\mu^{*})$.
 As a consequence $(\mu,x) \mapsto \mu^{-2}D_{0}(\mu,x)$ is $C^{0}$ in $\mu=0$ and 
 there exists 
$\mu_{0}\in{(0,\mu^{*}]}$ 
such that 
\begin{equation}
    \mu^{-2}D_{0}(\mu,x) \geq c>0,\text{ for any }\mu\in(0,\mu_{0}],\,\,x\in[0,L].
\end{equation}
}
Therefore, {from \eqref{eq:defD0},} there exists $\mu_{{a}}\in{(0,\mu_{0}]}$ such that, for any $\mu\in(0,\mu_{{a}})$,
\begin{equation}
    \text{det}(D) >0,\text{ for any }x\in[0,L].
\end{equation}
To conclude it suffices to note that, since $\mu>0$ is chosen such that $V_{x}\geq 0$, then 
\begin{equation}
    { 2} ({H^{*}}+\mu \tilde{q}_{2})\left({V^{*}_{x}}+\frac{f({H^{*}},{V^{*}})}{{{H^{*}}V^{*}}}\right)>0\text{ on }[0,L]
\end{equation}
and therefore {by Sylvester's criterion}, from the expression of $D$ given by \eqref{eq:expressionD}, $D$ is negative definite. This implies that there exists $\gamma > 0$ such that $\phi(\mu{, \gamma}) = \gamma Q + D$ is negative definite for any $x\in[0,L]$. This ends the proof of Proposition \ref{prop:inner}.   

\end{proof}

\begin{rmk}
{Note that} we cannot directly derive Proposition \ref{prop:inner} from \cite{bastin2017quadratic} and looking at the situation $\mu=0$. Indeed, when $\mu\rightarrow 0$, $(H^{*},V^{*})$ converges to a constant steady-state and the source term of the equation converges to $0$ and in this case the Lyapunov function found in \cite{bastin2017quadratic} does not work anymore. 
\end{rmk}

To conclude, let us now consider the full integral term \eqref{eq:lyap_I}:
\begin{equation}
    \mathcal{I} = \int_{0}^{L} y^T \phi(\mu, \gamma) y \; dx + 2 \int_0^L y_x^T QA y_x \; dx
\end{equation}

From Proposition~\ref{prop:inner}, $\phi(\mu, \gamma)$ is negative definite for all $\mu \in (0, {\mu_a})$ for some $\gamma, {\mu_a} > 0$. As a consequence, $y^T \phi(\mu, \gamma) y \leq 0$ for any $y(t,x)$, and thus $\int_0^L y^T \phi(\mu, \gamma) y \; dx \leq 0$.

Regarding the second integral, we have \[QA = \begin{pmatrix}0 & 0 \\ 0 & -4\mu(H^{*} + \mu \tilde{q_2})\end{pmatrix}\] where $H^* + \mu \tilde{q_2}$ is positive for $\mu \in (0, \mu_Q)$, with the same $\mu_Q > 0$ that was previously introduced. Thus $QA$ is semidefinite negative, leading to $\int_0^L y^T_x QA y_x \; dx \leq 0$. This concludes the proof that $\mathcal{I} \leq 0$.

\subsection{Boundary term}
\label{ssec:bound}

\begin{proof}[Proof of Proposition \ref{prop:boundary}]
Let us first expand the definitions of $Q$ and $A$ in the boundary term \eqref{eq:lyap_B}, which yields
\begin{equation}
    \mathcal{B} = \left[ -y^TQBy - 4\mu q_{2x} v^2 + 8 \mu q_2 vv_x \right]_0^L
\end{equation}
where $q_{2x} = (q_2)_x$. Furthermore, we can derive the following equalities from the linearized boundary conditions \eqref{eq:linear_bc}:
\begin{equation}
    \begin{split}
        y(t,0) &= \begin{psmallmatrix} 1 \\ -b_0 \end{psmallmatrix} h(t,0) \\
        y(t,L) &= \begin{psmallmatrix}1 \\ b_1\end{psmallmatrix} h(t,L) + \begin{psmallmatrix}0 \\ \mu c_1\end{psmallmatrix} v_x(t,L) \\
        v(t,0)^2 &= b_0^2 h(t,0)^2 \\
        v(t,L)^2 &= b_1^2 h(t,L)^2 + \mu^2 c_1^2 v_x(t,L)^2 + 2 \mu b_1 c_1 h(t,L) v_x(t,L)
    \end{split}
    \label{eq:linear_bc_csq}
\end{equation}
Using these equations \eqref{eq:linear_bc_csq}, the boundary term can be expanded as follows:
\begin{equation}
\begin{split}
    \mathcal B = & \; a_1(\mu) h(t,0)^2 + a_2(\mu) h(t,L)^2 + a_3(\mu) v_x(t,L)^2+ a_4(\mu) h(t,L) v_x(t,L)
\end{split}
\end{equation}
with
\begin{align}
    a_1(\mu) &= \begin{psmallmatrix}1 \\ -b_0\end{psmallmatrix}^T (QB)(0) \begin{psmallmatrix}1 \\ -b_0\end{psmallmatrix} + 4\mu b_0^2 q_{2x}(0) \\
    a_2(\mu) &= -\begin{psmallmatrix}1 \\ b_1\end{psmallmatrix}^T (QB)(L) \begin{psmallmatrix}1 \\ b_1\end{psmallmatrix} - 4\mu b_1^2 q_{2x}(L) \\
    a_3(\mu) &= -\begin{psmallmatrix}0 \\ \mu c_1\end{psmallmatrix}^T (QB)(L) \begin{psmallmatrix}0 \\ \mu c_1\end{psmallmatrix} + 8 \mu^2 c_1 q_2(L) - 4 \mu^3 c_1^2 q_{2x}(L) \\
    a_4(\mu) &= -2 \begin{psmallmatrix} 1 \\ b_1 \end{psmallmatrix}^T (QB)(L) \begin{psmallmatrix}0 \\ \mu c_1\end{psmallmatrix} + 8 \mu b_1q_2(L) - 8 \mu^2 b_1 c_1 q_{2x}(L) \label{eq:a4}
\end{align}
where in \eqref{eq:a4} we used the fact that $QB$ is symmetric \eqref{eq:defq1q2}.

We know from \cite{bastin2017quadratic} that under the conditions that $b_0 \in (b_0^-, b_0^+)$ and $b_1 \in \mathbb{R} \setminus [b_1^-, b_1^+]$ {(as defined in \eqref{eq:bc_b0} and \eqref{eq:bc_b1})}, 
we have $a_1(\mu=0) < 0 $ and $a_2(\mu=0) < 0$
where
\begin{align}
    a_1(\mu=0) &:= \begin{psmallmatrix}1 \\ -b_0\end{psmallmatrix}^T (QB)(0) \begin{psmallmatrix}1 \\ -b_0\end{psmallmatrix}, \\
    a_2(\mu=0) &:= -\begin{psmallmatrix}1 \\ b_1\end{psmallmatrix}^T (QB)(L) \begin{psmallmatrix}1 \\ b_1\end{psmallmatrix}.
\end{align}
Note that $a_1$ and $a_2$ are continuous functions of $\mu$, thus there exists ${\mu_1} > 0$ such that for any $\mu \in (0, {\mu_1})$, $a_1(\mu) < 0$ and $a_2(\mu) < 0$.
We assume in the following that $\mu\in(0,{\mu_1})$.
Let us now look at all the terms at $x=L$. We use the following notations for the coefficients of $QB(L)$:
\begin{equation}
\label{eq:notationM}
    QB(L) = \begin{pmatrix}
        \alpha & \beta \\ \beta & \gamma
    \end{pmatrix}
\end{equation}
For simplicity of the notations, we assume in the following that all functions are evaluated at $x=L$ and omit writing it. Developping all terms, we get:
\begin{align}
    a_2 &= -\alpha -2\beta b_1 -\gamma b_1^2 -4\mu b_1^2  q_{2x} \\
    a_3 &= \mu^2 (-\gamma c_1^2 +8 c_1 q_2 -4\mu c_1^2 q_{2x}) \\
    a_4 &= \mu (-2\beta c_1 -2\gamma b_1 c_1 +8 b_1 q_2 -8\mu b_1 c_1  q_{2x})
\end{align}
In order to show $\mathcal{B} \leq 0$ for any $y\in L^{2}(0,L)$, it suffices that (and in fact it is a necessary condition if one replaces the strict inequality by a large one)
\begin{equation}
    a_2h^2 + a_3v_x^2 + a_4 hv_x < 0, \quad \forall h,v\in\mathbb{R}.
    \label{eq:B_neg_a234}
\end{equation}
This holds if and only if
\begin{equation}
a_{2}<0\text{ and } \Delta_h := a_{4}^{2}-4a_{2}a_{3} <0.
\label{eq:young_condition}
\end{equation}
From the choice of ${\mu_1}$ we know that $a_2 < 0$ so the first condition is satisfied. As for the second condition, after simplifying we obtain:
\begin{equation}
    \Delta_h = \mu^2 ( d_1 c_1^2 + d_2 c_1 + d_3 )
\label{eq:poly_d}
\end{equation}
where $d_1$, $d_2$ and $d_3$ are independent of $c_1$ and given by
\begin{align}
\begin{split}
    d_1 &= -4 \alpha \gamma + 4 \beta^2 -16 \alpha \mu  q_{2x} \\
    & = -4 \det (QB(L)) -16 \alpha \mu  q_{2x}
\end{split} \\
    d_2 &= 32 q_2 (\alpha + \beta b_1) \\
    d_3 &= 64 b_1^2 q_2^2
\end{align}

Consider the case where $\mu=0$:
\begin{equation}
    \det QB(L) \stackrel{\mu=0}{=}
    \begin{vmatrix}g{V^{*}} & g{H^{*}} \\ g{H^{*}} & {H^{*}}{V^{*}}\end{vmatrix}
    = g{H^{*}}({V^{*}}^2 - g{H^{*}}) < 0
\end{equation}
because of the subcritical flow assumption~\eqref{eq:subcritical}. As a result, since $d_1$ is continuous in $\mu$, there exists ${\mu_2} > 0$ such that for any $\mu \in (0, {\mu_2})$, $d_1 > 0$.

Thus the polynomial $P_d(c_1) := d_1c_1^2 + d_2c_1 + d_3$ in variable $c_1$ from \eqref{eq:poly_d} has a positive leading coefficient,
and there exists $c_1 \in \mathbb R$ such that $P_d(c_1) < 0$ if and only if $\Delta_d \coloneqq d_2^2 - 4d_1d_3 > 0$.
We get:
\begin{equation}
    \Delta_d = -1024 q_2^2 \alpha a_2.
\end{equation}
We have seen before that $a_2 < 0$ given our choice of $b_1$ and for $\mu \in (0, {\mu_1})$, and there exists ${\mu_3}>0$ such that for any $\mu \in (0,{\mu_3})$, $\alpha > 0$ since $\alpha$ is continuous in $\mu$ and when $\mu=0$, $\alpha(\mu=0) = g{V^{*}} > 0$. Thus, $P_d(c_1) < 0$ for any $c_1 \in (c_{1,\mu}^-, c_{1,\mu}^+)$ with
\begin{equation}
    c_{1,\mu}^\pm \coloneqq \left. \frac{4 H^*\left(\zeta \left(\mu + 1\right) \left(H^* b_{1} + V^*\right) \pm \sqrt{\delta}\right)}{\zeta \left( (\mu+1) V^{*2} - \zeta\right)} \right|_{x=L},
\end{equation}
where
\begin{equation}
\begin{split}
    \delta(x) & \coloneqq  V^* \zeta \left(\mu + 1\right)^{2} (H^* b_{1}^{2} \left(\mu + 1\right) \left(H^* V^* - 4 H_{x}^* \mu\right)  + 2 H^* b_{1}  \left(2 H_{x}^* b_{1} \mu \left(\mu + 1\right) + \zeta\right) + V^* \zeta )
\end{split}
\end{equation}
and
\begin{equation}
\begin{split}
    \zeta(x) &\coloneqq gH^* + \mu \left(H^* + \tilde{q_1} \right) \\ &= gH^* + \mu \left(H^* + g - 4(1+\mu) \frac{V_x^*}{H^*} \right).
\end{split}
\end{equation}
We can note that there exists ${\mu_4} > 0$ such that for any $\mu \in (0, {\mu_4})$, if $c_1 \in (c_1^-, c_1^+) \coloneqq (c_{1,\mu=0}^-, c_{1,\mu=0}^+)$ {(as defined in~\eqref{eq:bc_c1}),}
then $c_1 \in (c_{1,\mu}^-, c_{1,\mu}^+)$, since 
 $c_{1,\mu}^\pm$ are both continuous in $\mu$. Thus, for $\mu \in (0, {\mu_4})$, $P_d(c_1)<0$ if $c_1 \in (c_1^-, c_1^+)$.

Consequently, {selecting $\mu_{b} = \min({\mu_1,\mu_2,\mu_3,\mu_4})$ for $\mu \in (0, \mu_{b})$}, the second condition in \eqref{eq:young_condition} is satisfied, which concludes the proof that $\mathcal{B} \leq 0$ for adequately chosen values of $b_0$, $b_1$ and $c_1$.
\end{proof}

\section{{Numerical Simulations}}

{
We present in Figure~\ref{fig:fig1num} a numerical illustration of Theorem~\ref{th:main}. The simulation was performed with an Implicit-Explicit (IMEX) Scheme, using $L = 1\,\text{km}$, $T = 3500\,\text{s}$, $g = 9.81\,\text{m}\cdot\text{s}^{-2}$, $\mu = 0.001\,\text{m}^{2}\cdot\text{s}^{-1}$, and $\kappa = 0.002\,\text{m}\cdot\text{s}^{-1}$. The values $H^{*}$ and $V^{*}$ are obtained as the unique (time-stationary) solution of  \eqref{eq:sv_nonlinear_1}–\eqref{eq:sv_nonlinear_2} with initial condition $V(0) = 1\,\text{m}\cdot\text{s}^{-1}$, $H(0) = 4\,\text{m}$, and $V_{x}(0)=0$, solved using an LSODA solver~\cite{hindmarsh1983odepack}. The boundary coefficients are chosen as $b_{0}= g/V^{*}(0)$ and
\begin{equation}
b_{1} = 
g \sqrt{ \frac{1}{V^*(L)^2} - \frac{1}{g H^*(L)} } 
,\;\; c_{1} =\frac{4 \left( V^*(L) + b_1 H^*(L)\right)}{V^*(L)^{2} - g H^*(L)}.
\end{equation}

The discretization parameters are $dt = 0.033\,\text{s}$ and $dx = 0.5\,\text{m}$, chosen to satisfy the CFL condition associated with the scheme. The initial perturbations are set as $h(x) =0.01\cos(20x+15)$ and $v(x) = 0.01\cos(x)$. We plot the evolution of the $L^{2}$ norm of the solution in both linear and logarithmic scales. Note that the logarithmic scale is shown on a much longer time horizon to highlight the exponential decay.}

\begin{figure}
    \centering
    \includegraphics[width=0.99\linewidth]{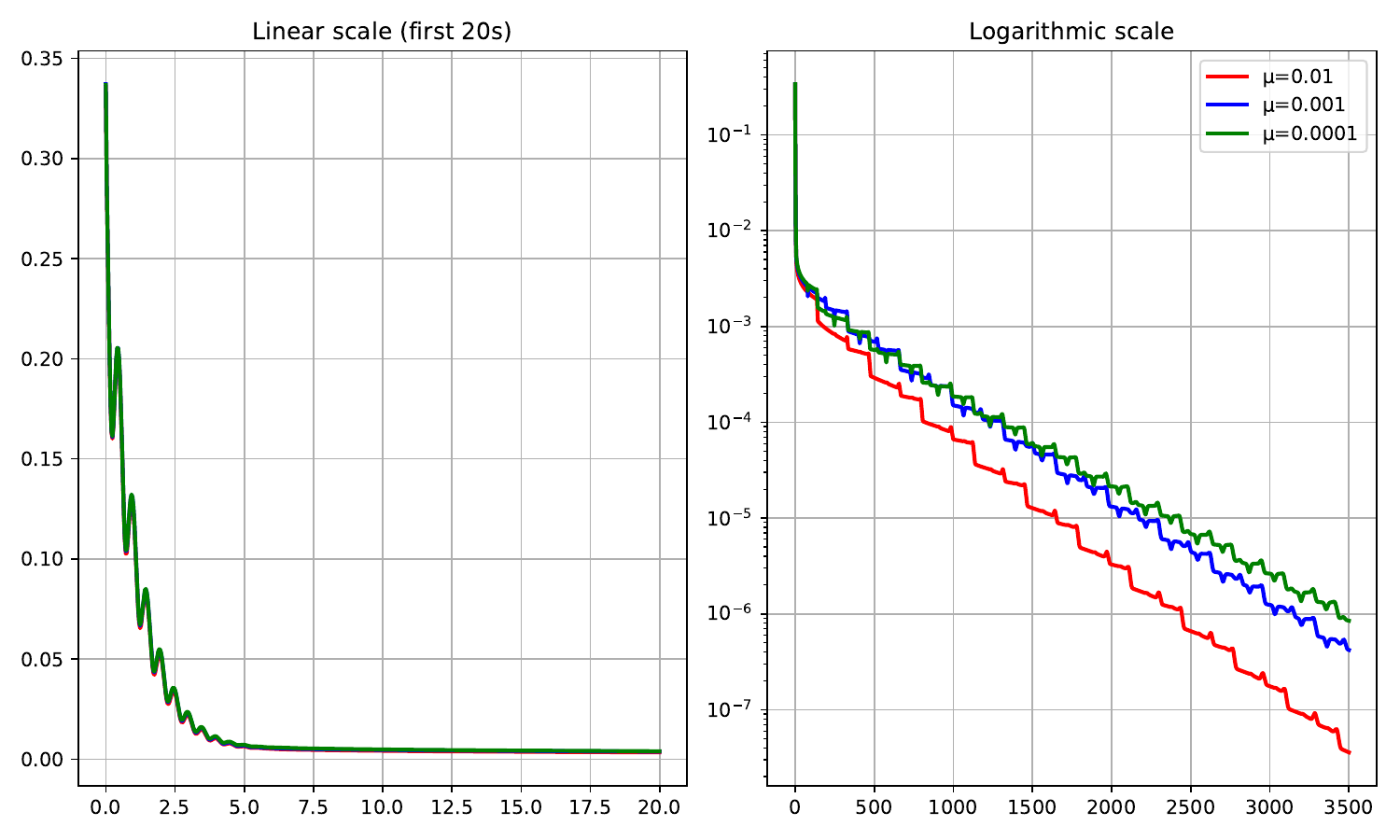}
\caption{{Time evolution of the $L^{2}$-norm of the solution $(h,v)$ of \eqref{eq:linear_system} for various viscosity values, on a \textbf{linear} (left) and \textbf{logarithmic} (right) scale.}}
    \label{fig:fig1num}
\end{figure}

\section{Conclusion}
\label{sec:conclusion}

In this paper, we showed the exponential stability of the steady-states for the {linearized} viscous Saint-Venant equations, by constructing a quadratic Lyapunov functions and carefully designing the boundary feedback controls. Quadratic Lyapunov functions are often robust to non-linearities, which is promising for the stability of the non-linear system, however that remains an open question.

\appendix

{
\section{Proof that $Q$ must be diagonal in physical coordinates}
\label{app:proof_Q_diag}

\begin{proof}[Proof of Proposition~\ref{prop:Q_diag}]
    First, let us note that we can without loss of generality assume that $Q$ is symmetric, since for any matrix $Q$, $y^T Q y = y^T \left( \frac{Q+Q^T}{2} \right) y$. Thus, let us assume a general symmetric form of $Q$:
    \begin{equation}
        Q = \begin{pmatrix}
            q_1 & q_3 \\ q_3 & q_2
        \end{pmatrix}
    \end{equation}
    then
    \begin{equation}
        {W}(h,v) = \int_0^L (q_1 h^2 + q_2 v^2 + 2q_3hv ) \; dx.
    \end{equation}
    Assuming that ${W}$ is a Lyapunov function, it must be positive for any $(h,v) \neq (0,0)$, which implies $q_1 > 0$ and $q_2 > 0$.
    
    The last term in our integral term $\mathcal{I}$ in \eqref{eq:lyap_I} is 
    \begin{align*}
        \mathcal{I}_{y_x} &= 2\int_0^L y^T_x QA y_x \; dx \\
        &= 2\int_0^L \begin{pmatrix}h_x \\ v_x\end{pmatrix}^T \begin{pmatrix}
            0 & -4\mu q_3 \\ 0 & -4\mu q_2
        \end{pmatrix} \begin{pmatrix}h_x \\ v_x\end{pmatrix} \; dx \\
        &= -8\mu \int_0^L (q_3 h_xv_x + q_2 v_x^2) \; dx
    \end{align*}
    In order to be able to prove that $\mathcal{I} + \mathcal{B} \leq 0$ for any $h,v$, we need $\mathcal{I}_{y_x} \leq 0$ and consequently that $q_{3}=0$. Indeed, if there exists $x_{1}\in [0,L]$  such that $q_{3}(x_{1})\neq0$, since $q_{3}$ is continuous there exists a non-empty interval $l$ such that $q_{3}>0$ on $l$, and we can consider a sequence of compactly-supported functions $h^{n}$ and $v^{n}$ such that the integral of $-8\mu q_{3}h_{x}v_{x}$ diverges to $+\infty$ while the integral of $-8\mu q_{2}v_{x}^{2}$ remains bounded, as well as the remaining term $\mathcal{I}-\mathcal{I}_{y_x}$ (note that $\mathcal{B}=0$ for compactly supported functions).
    {See also \cite[Theorem 3.2]{hayat2018exponentialstabilitygeneral1d} or \cite[Lemma 1]{bastin2011boundary} for similar arguments.}
\end{proof}
}

\section{Derivation of the Lyapunov function}
\label{app:lyapunov_derivation}

By substitution of $y_t$ and since $Q$ is diagonal, we have:
\begin{align}
    \dot{\bf {W}} + \gamma {\bf {W}} &= \int_0^L (2 y^T Q y_t + \gamma y^T Q y) \; dx \\
    &= -2 \int_0^L y^T Q (Ay_{xx} + By_x + Cy) \; dx +  \gamma \int_0^L y^T Q y \; dx
\end{align}
Given a symmetric matrix $M \in \mathbb{R}^{2 \times 2}$, we have $[y^TMy]_0^L = \int_0^L (y^TMy)_x \; dx = \int_0^L (y^T_xMy + y_TM_xy + y_TMy_x) \; dx$. Since $M$ is symmetric, $y^T_xMy = y^TMy_x$ and we get the following identity:
\begin{equation}
    \int_0^L y_x^T M y \; dx = \frac{1}{2} [y^TMy]_0^L - \frac{1}{2} \int_0^L y_TM_xy \; dx
    \label{eq:A_identity}
\end{equation}
We first derive the term in $y_{xx}$ with an integration by parts followed by applying identity \eqref{eq:A_identity} ($QA$ being diagonal thus symmetric) and the fact that $(QA)_x = Q_xA$ and $(QA)_{xx} = Q_{xx}A$ since $A_x=0$:
\begin{equation}
\begin{split}
    \int_0^L y^T QA y_{xx} \; dx &= \left[y^T QA y_x\right]_0^L - \int_0^L y^T_x QA y_x \; dx - \int_0^L y^T Q_xA y_x\;  dx
\end{split}
\end{equation}
where
\begin{equation}
    - \int_0^L y^T Q_xA y_x\;  dx = - \frac 12 \left[ y^T Q_xAy \right]_0^L + \frac 12 \int_0^L y^T Q_{xx}A y\;  dx.
\end{equation}
We then derive the term in $y_x$ again using identity \eqref{eq:A_identity}, $QB$ being symmetric by construction \eqref{eq:defq1q2}:
\begin{align}
    \int_0^L y^T QB y_x \; dx &= \frac 12 \left(\left[ y^T QB y \right]_0^L -  \int_0^L y^T (QB)_x y \; dx\right)
\end{align}
Finally, note that $y^TQCy = y^T(QC)^Ty$ thus $2y^TQCy = y^T(QC+(QC)^T)y$.
Putting it all together, we get:
\begin{align}
\begin{split}
    \dot{\bf {W}} + \gamma {\bf {W}} &= \left[ y^T (Q_xA - QB) y -2 y^T QA y_x\right]_0^L \\ &+ \int_0^L y^T \left(\gamma Q - (QC+(QC)^T)  -Q_{xx}A + (QB)_x \right) y \; dx + 2\int_0^L y^T_x QA y_x \; dx
\end{split}
\end{align}
which concludes the derivation.

\section{Existence of the steady-state and asymptotic estimation of $V^{*}_{x}$ {and $V^{*}_{xx}$}}
\label{app:estimVx}
In this Appendix we show Proposition \ref{prop:estimVx}.
Our goal is to show that for any positive $H_{0}$, $V_{0}$ such that $gH_{0}-V_{0}^{2}>0$, there exists a unique solution $(H^{*},V^{*})\in C^{2}([0,L];(0,+\infty))$ to
\begin{equation}
\label{eq:sys1HV}
\begin{cases}
&(H^{*}V^{*})_x = 0,\\
&V^{*}V_x^{*} + gH_x^{*} + \frac{f(H^{*}, V^{*})}{H^{*}} - 4 \mu  \frac{(H^{*}V_x^{*})_x}{H^{*}} = 0,\\
&H^{*}(0) = H_{0},\;\;V^{*}(0)=V_{0},\;\; 
V_{x}^{*}(0)=0,
\end{cases}
\end{equation}
where $f$ is given by \eqref{def:f}.
Denoting $Q^{*} = H^{*}V^{*}$ this is equivalent to show that for any positive $Q_{0}$, $V_{0}$ such that $g Q_{0}-V_{0}^{3}>0$, there exists a unique solution $(Q^{*}, V^{*})\in C^{2}([0,L];(0,+\infty))$ to
\begin{equation}
\label{eq:sys2QV}
\begin{cases}
Q_{x}^{*}= 0\\
V^{*}V_{x}^{*} - \frac{gQ^{*}}{{V^{*}}^{2}}V_{x}^{*} + 3\mu\frac{\tilde{f}(V^{*})V^{*}}{Q^{*}} - 4 \mu V_{xx}^{*} +4\mu \frac{(V_{x}^{*})^{2}}{V^{*}}= 0,\\
Q^{*}(0) = Q_{0},\;\;V^{*}(0) = V_{0},\;\ V^{*}_{x}(0) = 0,
\end{cases}
\end{equation}
where {we used the fact that the first equation can be rewritten as $H^{*}_x = - \frac{Q^{*} V^{*}_{x}}{(V^*)^{2}}$, and}
\begin{equation}
\label{eq:deftildef}
\tilde{f}(V^{*}) = \frac{{V^{*}}^{2}\kappa}{3\mu V^{*}+\kappa Q^{*}}.
\end{equation}
The first equation, {along with the initial condition}, reduces to $Q^{*}=Q_{0}$, thus the system reduces to
\begin{equation}
\label{eq:sys3}
\begin{cases}
V^{*}V_{x}^{*} - \frac{gQ_{0}}{{V^{*}}^{2}}V_{x}^{*} + 3\mu\frac{\tilde{f}(V^{*})V^{*}}{Q_{0}} - 4 \mu V_{xx}^{*} +4\mu \frac{(V_{x}^{*})^{2}}{V^{*}}= 0,\\
V^{*}(0) = V_{0}, \;\;V^{*}_{x}(0) = 0.
\end{cases}
\end{equation}
The difficulty is that this system is singular: the terms in $\mu$ are also the terms with the highest order derivatives. Asymptotic behavior of singular systems are usually difficult to handle 
and nothing guarantees a priori, that there exists a unique solution in $C^{2}([0,L];(0,\infty))$ for $\mu>0$ small with an estimate of the form
\begin{equation}
V^{*} = V_{\text{inviscid}} + O(\mu),
\end{equation}
where $V_{\text{inviscid}}$ is the solution of the system when $\mu=0$ (in this case $V_{\text{inviscid}}(x) = V_{0}$, for any $x\in[0,L]$). Denoting $(y,z)  = (V,V_{x})$, the system \eqref{eq:sys3} becomes
\begin{equation}
\label{eq:sys4}
    \begin{cases}
        &y_{x} = z\\
        & \mu z_{x} = \frac{1}{4}z\left(y-g \frac{Q_{0}}{y^{2}}\right)+\frac{3}{4}\mu \frac{\tilde{f}(y)y}{Q_{0}}+\mu \frac{z^{2}}{y}.
    \end{cases}
\end{equation}
We are going to show the following Lemma:
\begin{lem}
\label{lem:appendix}
Let $Q_{0}>0$, $\varepsilon>0$, $c_{y}>0$, $C_{z}>0$ and
\begin{equation}
\begin{split}
\Omega(\varepsilon) := \Big\{&(y,z)\in (0,+\infty)\times\mathbb{R} \quad \Big| \quad \frac{gQ_{0}}{y^{2}}-y>\varepsilon,\; y\geq c_{y},\; |z|\leq C_{z} \Big\}.
\end{split}
\end{equation}
There exists $\mu^{*}>0$ such that for any $\mu\in(0,\mu^{*})$ and any $y_{0},z_{0}\in \Omega(2\varepsilon)$, there exists a unique solution $(y,z)\in C^{1}([0,L];\Omega(\varepsilon))$ to \eqref{eq:sys4} with initial condition $(y_{0},z_{0})$. Moreover there exists $C_{1}>0$ depending only on $\varepsilon$, $Q_{0}$, $c_{y}$ and the parameters of the system such that 
\begin{equation}
\label{eq:estimyz}
\begin{split}
|y(x)-y_{0}|&\leq \frac{4\mu}{\varepsilon}\left(1-e^{-\frac{\varepsilon x}{4\mu}}\right)|z_{0}| + C_{1}\frac{4\mu}{\varepsilon}\left(x-\frac{4\mu}{\varepsilon}(1-e^{-\frac{\varepsilon x}{4\mu}}))\right),\\
|z(x)|&\leq  |z_{0}| e^{-\frac{\varepsilon x}{4\mu}}+ C_{1}\frac{4\mu}{\varepsilon}(1-e^{-\frac{\varepsilon x}{4\mu}}),
\end{split}
\end{equation}
{and for any $x_{0}\in[0,L]$ such that $z(x_{0})=0$, $z_{x}(x_{0})>0$.}
\end{lem}
\begin{proof}
Let $Q_{0}>0$, $\varepsilon>0$, $c_{y}>0$, $C_{z}>0$ and let $\mu\in(0,\mu^{*})$ with $\mu^{*}$ to be chosen later on. {Note that on $\Omega(\varepsilon)$, the function $(y,z)\mapsto (z, z(y-gQ_{0}/y^{2})/4+3\mu\tilde{f}(y)y/(4Q_{0})+\mu z^{2}/y)$ is of class $C^{1}$ and hence locally Lipschitz. Thus from} the {(local)} Cauchy-Lipschitz theorem, 
there exists a unique maximal solution $(y, z)$ in $\Omega(\varepsilon)$ to \eqref{eq:sys4} defined 
on $[0,L^{*})$ with $L^{*}=+\infty$ or $L^{*}\in(0,+\infty)$ and in {the second} case either
\begin{equation}
\label{eq:boundslim}
\begin{split}
&\lim\limits_{x\rightarrow L^{*}}(gQ_{0}/y^{2}(x)-y(x)) =\varepsilon\text{ with }y(x)> 0\\
\text{ or }&\lim\limits_{x\rightarrow L^{*}}y(x) = c_{y}\\
\text{ or }&\lim\limits_{x\rightarrow L^{*}}|z(x)| = C_{z}
\end{split}
\end{equation}
Note that $gQ_{0}-y^{3}-\varepsilon y^{2}$ has a unique non-negative (and actually positive) root (because it is positive at $y=0$, negative when $y \to +\infty$ and its derivative is negative for $y > 0$) that we denote by $y_{1,\varepsilon}>0$, this means that 
\begin{equation}
\label{eq:defy1}
\lim\limits_{x\rightarrow L^{*}}(gQ_{0}/y^{2}(x)-y(x)) =\varepsilon \iff
   \lim\limits_{x\rightarrow L^{*}}y(x) = y_{1,\varepsilon}.
\end{equation}
In any cases, on $[0,L^{*})$, we can define 
\begin{equation}
\label{eq:defg}
{\Gamma(x)} =\frac{gQ_{0}}{y^{2}(x)}-y(x),
\end{equation}
and we have
\begin{equation}
    \mu z'(x) = -\frac{1}{4}{\Gamma(x)}z(x)+\mu\left(\frac{3}{4}\frac{\tilde{f}(y(x))y(x)}{Q_{0}}+\frac{z(x)^{2}}{y(x)}\right),
\end{equation}
hence, by Duhamel's formula,
\begin{equation}
\begin{split}
    z(x) &= z_{0} e^{-\frac{1}{4\mu}\int_{0}^{x}{\Gamma(s)}ds}+  \int_{0}^{x}e^{-\frac{1}{4\mu}\int_{s}^{x}{\Gamma(\tau)}d\tau}\left(\frac{3}{4}\frac{\tilde{f}(y(s))y(s)}{Q_{0}}+\frac{z(s)^{2}}{y(s)}\right) ds.
\end{split}
\end{equation}
Since $(y(s),z(s))\in \Omega(\varepsilon)$ for any $s\in [0,x]\subset[0,L^{*})$ and from \eqref{eq:defy1},
\begin{equation}
\label{eq:estimz}
\begin{split}
    |z(x)| &\leq |z_{0}| e^{-\frac{\varepsilon x}{4\mu}}+C_{1}\int_{0}^{x}e^{-\frac{\varepsilon (x-s)}{4\mu}} ds \\&= |z_{0}| e^{-\frac{\varepsilon x}{4\mu}}+ C_{1}\frac{4\mu}{\varepsilon}(1-e^{-\frac{\varepsilon x}{4\mu}}),
\end{split}
\end{equation}
where $C_{1} = 3Q_{0}^{-1}\max\limits_{[c_{y},y_{1,\varepsilon}]}|\tilde{f}(y)y|/4+ C_{z}^{2}c_{y}^{-1}$. 
Therefore, using the first equation of \eqref{eq:sys4},
\begin{equation}
\label{eq:estimy}
\begin{split}
|y(x) - y_{0}| &\leq \int_{0}^{x} |z_{0}| e^{-\frac{\varepsilon s}{4\mu}}+ C_{1}\frac{4\mu}{\varepsilon}(1-e^{-\frac{\varepsilon s}{4\mu}}) ds\\
&= \frac{4\mu}{\varepsilon}\left(1-e^{-\frac{\varepsilon x}{4\mu}}\right)|z_{0}|  + C_{1}\frac{4\mu}{\varepsilon}\left(x-\frac{4\mu}{\varepsilon}(1-e^{-\frac{\varepsilon x}{4\mu}}))\right).
\end{split}
\end{equation}
We see from \eqref{eq:estimz}--\eqref{eq:estimy} that we can choose $\mu^{*}$, depending on $C_{z}$, $c_{y}$ and $\varepsilon$ such that $L^{*}>L$ (otherwise we get a contradiction because we never reach \eqref{eq:boundslim}). {In fact, this characterizes $\mu^{*}$ and it can be computed as the limiting constant at which \eqref{eq:boundslim} is reached.} 
This shows the existence (and unicity) of the solution $(y,z)\in C^{1}([0,L];\Omega(\varepsilon))$. In addition the estimates \eqref{eq:estimz}--\eqref{eq:estimy} hold, {which is exactly \eqref{eq:estimyz}. Finally, assume that for $x_{0}\in[0,L]$, $z(x_{0})=0$, since $(y(x_{0}),z(x_{0}))\in \Omega(\varepsilon)$, $y(x_{0})\geq c_{y}>0$ and, from \eqref{eq:sys4} and the expression of $\tilde{f}$ given by \eqref{eq:deftildef} we deduce directly that $z_{x}(x_{0})>0$,} which ends the proof of Lemma \ref{lem:appendix}.
\end{proof}

Proposition \ref{prop:estimVx} follows from Lemma \ref{lem:appendix}. Indeed, for any steady-state $(H_{0},V_{0})\in (0,+\infty)^{2}$ such that
\begin{equation}
gH_{0}-V^{2}_{0}>0,
\end{equation}
we can denote $Q_{0} = H_{0}V_{0}$ which is a constant, $y_{0}:=V_{0}$, $c_{y}:=V_{0}/2$, and $z_{0} = 0$ and there exists $\varepsilon>0$ such that $gQ_{0}/(V_{0})^{2}-V_{0}>2\varepsilon$. From Lemma \ref{lem:appendix} and using \eqref{eq:sys1HV}--\eqref{eq:sys4}, there exists $\mu^{*}>0$ such that for any $\mu\in(0,\mu^{*})$ there exists a unique solution $(H^{*},V^{*})\in C^{2}([0,L];(0,+\infty)^{2})$ 
to \eqref{eq:sys1HV} that satisfies in addition,
\begin{equation}
    gH^{*}(x)-(V^{*}(x))^{2}>\varepsilon V^{*}(x)>\varepsilon V^{*}(0)/2,\;\forall x\in[0,L],
\end{equation}
and 
{
\begin{equation}
|V^{*}(x)-V_{0}|\leq C_{1}L\frac{4\mu}{\varepsilon}
\end{equation}
and}
\begin{equation}
\label{eq:vx_bounded}
    |V_{x}^{*}(x)|\leq 
    |V_{x}^{*}(0)|e^{-\varepsilon x/4\mu} +
    \mu \frac{4 C_{1}}{\varepsilon}.
\end{equation}
Since 
$V_{x}^{*}(0)=0$, we get that $\|V_{x}^{*}\|_{L^{\infty}}\leq \mu 4 C_{1}/\varepsilon$. From \eqref{eq:sys2QV} we deduce that
\begin{equation}
|V_{xx}^{*}(x)|< C_{2},
\end{equation}
where $C_{2}$ is a constant that only depends on $H_{0}$ and $V_{0}$. 
{Moreover, we can show that $V_{x}^{*} \geq 0$ for all $x \in [0,L]$. Indeed, $V_{x}^{*}(0)=0$ and from Lemma \ref{lem:appendix}, $V_{xx}^{*}(0)>0$ which means that there exists $l_{1}>0$ such that $V_{x}^{*}>0$ on $[0,l_{1}]$. Suppose by contradiction that there exists an $x_1 \in [l_{1},L]$ such that $V_{x}^{*}(x_1) < 0$. Since $V^{*}$ is continuous, there must be a point $x_2 \in (l_{1}, x_1)$ such that $V_{x}^{*}(x_2) = 0$ and $V_{xx}^{*}(x_{2}) < 0$, but this is in contradiction with Lemma \ref{lem:appendix}.}

This shows the first part of Proposition \ref{prop:estimVx} and it only remains to show \eqref{eq:estimVmu}. To do so, we can perform as in Lemma \ref{lem:appendix}.
Indeed, let $\hat{z}= 3\mu \frac{\tilde{f}(y)y}{Q_{0}\left(g \frac{Q_{0}}{y^{2}}-y\right)}$. From Lemma \ref{lem:appendix} we know that $(y,z)$ with $y=V^{*}$, $z=V_{x}^{*}$ exists in $C^{1}([0,L];\Omega(\varepsilon))$ for our choice of $\mu$ and \eqref{eq:estimyz} holds.
Observe that, from \eqref{eq:defg},
\begin{equation}
\frac{3}{4}\frac{\tilde{f}(y)y}{Q_{0}} = \frac{1}{4\mu}{\Gamma(x)}\hat{z}.
\end{equation}
Thus, we have from \eqref{eq:sys4}
\begin{equation}
    (z-\hat{z})' = -\frac{1}{4\mu}(z-\hat{z}){\Gamma(x)}+\frac{z^{2}}{y}+z\mu\partial_{y}\left(\frac{3\tilde{f}(y)y}{Q_{0}\left(g \frac{Q_{0}}{y^{2}}-y\right)}\right),
\end{equation}
thus, by Duhamel's formula,
\begin{equation}
\begin{split}
\label{eq:estimhat}
&z(x)-\hat{z}(x) = (z_{0}-\hat{z}(0))e^{-\frac{1}{4\mu}\int_{0}^{x}{\Gamma(s)}d{s}}+\int_{0}^{x} e^{-\frac{1}{4\mu}\int_{s}^{x}{\Gamma(\tau)}d\tau}\\& 
\times\left[\frac{z(s)^{2}}{y(s)}+z(s)\partial_{y}\left(\frac{3\mu\tilde{f}(y(s))y(s)}{Q_{0}\left(g \frac{Q_{0}}{y(s)^{2}}-y(s)\right)}\right)\right]ds.
\end{split}
\end{equation}
As a consequence, from \eqref{eq:estimyz} with $z_{0}=0$, performing as in \eqref{eq:estimz} {and also using \eqref{eq:vx_bounded},}
\begin{equation}
    |z(x)-\hat{z}(x)|\leq |\hat{z}(0)|e^{-\frac{\varepsilon x}{4\mu}} + \frac{C}{\varepsilon}\mu^{3}.
\end{equation}
From the definition of $\hat{z}$, we have
\begin{equation}
    |z(x)-\hat{z}(x)|\leq \frac{C}{\varepsilon}\mu e^{-\frac{\varepsilon x}{4\mu}} + \frac{C}{\varepsilon}\mu^{3},
\end{equation}
where $C$ is a constant that may change between lines but only depends on $y_{0}$ and the parameters of the system and not on $\mu$.
As a consequence, using \eqref{eq:estimhat}
\begin{equation}
z_{x} = \hat{z}_{x} +\frac{1}{4\mu}{\Gamma(x)}\hat{z}(0)e^{-\frac{1}{4\mu}\int_{0}^{x}{\Gamma(s)}d{s}}+O(\mu^{2}).
\end{equation}

Translating $z$ and $y$ as $V_{x}$ and $V$ respectively gives the estimates \eqref{eq:estimVmu}. This ends the proof of Proposition \ref{prop:estimVx}.

\section{Well-posedness of the linear system}
\label{app:well_posedness_linear}

We consider the linear operator
\begin{equation}
\mathcal{A} = -(A\partial_{xx}+ B(x)\partial_{x}+C(x) \text{Id})
\end{equation}
where $\text{Id}$ denotes the identity operator. $\mathcal{A}$ is an operator in the space $L^2([0,L]) \times L^2([0,L])$ with domain
\begin{equation}
\begin{split}
D(\mathcal{A}) &= \{f = (h,v)^{T}\in L^{2} \;|\; \mathcal{A}f\in L^{2},
 v(0) = -b_{0}h(0),\;v_{x}(0)=0,
 v(L) = b_{1}h(L)+\mu c_{1}v_{x}(L)\}.
\end{split}
\end{equation}
We claim the following:
\begin{itemize}
\item[(i)] $D(\mathcal{A})$ is dense in $L^{2}$,
\item[(ii)] $\mathcal{A}$ is dissipative, \emph{i.e.} $\text{Re} \langle u, \mathcal{A}u \rangle \leq 0$ for every $u\in D(\mathcal{A})$,
\item[(iii)] There exists $\lambda_{0}>0$ such that $\mathcal{A}-\lambda_{0} \text{Id}$ is surjective.
\end{itemize}

If these three conditions (i)-(iii) are satisfied, then by the Lumer-Phillips theorem~\cite{lumer1961dissipative,renardy2004introduction}, $\mathcal{A}$ generates a $C^0$-semigroup. Then, the initial value problem
\begin{equation}
    \begin{cases}
        \displaystyle \frac{d}{dt} \begin{pmatrix}h \\ v\end{pmatrix}(t) = \mathcal{A} \begin{pmatrix}h \\ v\end{pmatrix}(t) \qquad \text{for } t \geq 0, \\
        v(0) = -b_{0}h(0),\;\; 
        v(L) = b_{1}h(L)+\mu c_{1}v_{x}(L), \\ v_{x}(0)=0
    \end{cases}
\end{equation}
is well-posed (see \cite{engel2000oneparameter} Corollary 6.9), thus our linear Saint-Venant with viscosity system \eqref{eq:linear_system} is well-posed.

We now proceed to proving (i)-(iii).

(i) follows directly from the density of $C^{2}$ with compact support in $L^{2}$ since any function that is $C^{2}$ with compact support in $(0,L)$ belongs to $D(\mathcal{A})$. 

(ii) follows from the fact that we derived a Lyapunov function in the linear case. Let us consider the associated inner product $\langle u, v \rangle = \int_{0}^{L} u(x)^T Q(x) v(x) \, dx$ for $u,v \in L^2([0,L])^2$, with $Q$ defined in \eqref{eq:defQ}. Consider $u = (h, v)^T \in D(\mathcal{A})$, since $Q = Q^T$ we have 
\begin{equation}
    \text{Re} \langle u, \mathcal{A}u \rangle = \langle u, u_t \rangle = \int_0^L u_t^T Q u \, dx = \frac{1}{2} \dot{\bf{{W}}} \leq 0
\end{equation} where the last inequality directly follows from Proposition~\ref{prop:lyap_3_props}.

Let us now look at (iii). This is equivalent to show that there exists $\lambda_{0}>0$ such that for any $f=(f_{1},f_{2})^{T}\in L^{2}$ there exists $u=(h,v)^{T}\in D(\mathcal{A})$ such that
\begin{equation}
\mathcal{A}u - \lambda_0u = f.
\end{equation}
This is equivalent to saying that the following Cauchy problem admits a solution on $x \in [0, L]$:
\begin{equation}
\begin{cases}
-(A u_{xx} + B(x) u_x) - (C(x) + \lambda_0 I) u = f(x) \\
v(0) = -b_{0}h(0),\;\; v(L) = b_{1}h(L)+\mu c_{1}v_{x}(L) \\ v_{x}(0)=0
\end{cases}
\end{equation}
Expanding the ODE using the definitions of $A$~\eqref{eq:linear_A}, $B$~\eqref{eq:linear_B} and $C$~\eqref{eq:linear_C}, we get
\begin{equation}
\label{eq:wp1}
\begin{split}
&h_{x}V^{*}+v_{x}H^{*}+H_{x}^{*}v+V_{x}^{*}h+\lambda_{0} h = -f_{1}\\
&\left(V^{*}-4\mu\frac{H_x^{*}}{H^{*}}\right)v_{x}+\left(g-4\mu \frac{V_x^{*}}{H^{*}}\right)h_{x}+V^{*}_{x}v\\&\;\;+S_{1}(x)h+S_{2}(x)v +\lambda_{0}v-4\mu v_{xx} = -f_{2}
\end{split}
\end{equation}
where $S_1(x)$ and $S_2(x)$ are the corresponding terms in $C$. Setting $v_{2} = v_{x}$, we obtain
\begin{equation}
\begin{split}
h_{x}V^{*}+v_{2}H^{*}+H_{x}^{*}v+V_{x}^{*}h+\lambda_{0} h &= -f_{1}\\
\left(V^{*}-4\mu\frac{H_x^{*}}{H^{*}}\right)v_{2}+\left(g-4\mu \frac{V_x^{*}}{H^{*}}\right)h_{x}+V^{*}_{x}v +S_{1}(x)h+S_{2}(x)v +\lambda_{0}v-4\mu v_{2x} &= -f_{2}.
\end{split}
\end{equation}
{The total system} is of the form
\begin{equation}
\label{eq:wp2}
M_{1}\begin{pmatrix}h\\v_{2}\\v\end{pmatrix}_{x} + M_{2}\begin{pmatrix}h\\v_{2}\\v\end{pmatrix}= {-} \begin{pmatrix}f_{1}\\f_{2}\\0\end{pmatrix}
\end{equation}
with
\begin{equation}
\begin{split}
M_{1}=& \begin{pmatrix}
V^{*} & 0 & 0\\
g-4\mu\frac{V_{x}^{*}}{H^{*}} & -4\mu & 0\\
0& 0 & 1
    \end{pmatrix},\\
    M_{2} =&\begin{pmatrix}
    V_{x}^{*}+\lambda_{0} & H^{*} & H_{x}^{*}\\
    S_{1} & V^{*}-4\mu \frac{H_{x}^{*}}{H^{*}} & S_{2}+\lambda_{0}+V_{x}^{*}\\
    0 & -1 & 0
    \end{pmatrix}
    \end{split}
\end{equation}
with
 $v(0) = -b_{0}h(0),\;v(L) = b_{1}h(L)+\mu c_{1}v_{2}(L),\; v_{2}(0) = 0.$
Let us diagonalize $M_{1}$, with $P^{-1}M_{1}P=\Lambda(x)$, and let $z = P^{-1}(h,v_{2},v)^{T}$ and ${\psi} = P^{-1}(f_{1},f_{2},0)^{T}$. Then \eqref{eq:wp2} is equivalent to
\begin{equation}
\Lambda \partial_{x}z + P^{-1}M_{2}P z - \Lambda (\partial_{x}P^{-1})Pz = {-} {\psi},
\end{equation}
where
\begin{equation}
\begin{split}
    P = \begin{pmatrix}
        0 & \frac{1}{\Delta(x)} & 0 \\ 0 & 1 & 1 \\ 1 & 0 & 0
    \end{pmatrix}, \;
    P^{-1} = \begin{pmatrix}
        0 & 0 & 1 \\ \Delta(x) & 0 & 0 \\ -\Delta(x) & {1} & 0
    \end{pmatrix},\\
    \Lambda = \begin{pmatrix}
        1 & 0 & 0 \\ 0 & V^* & 0 \\ 0 & 0 & -4 \mu
    \end{pmatrix}, \;
    z = \begin{pmatrix}
        v \\ h \Delta(x) \\ -h \Delta(x) + v_2
    \end{pmatrix}
    \end{split}
\end{equation}
with \begin{equation}\Delta(x) = \frac{1}{4\mu + V^*}\left( g-4\mu\frac{V_x^*}{H^*} \right).\end{equation}
We set
\begin{equation}
    M_{3} = P^{-1}M_{2}P-\Lambda(\partial_{x}P^{-1})P
\end{equation}
yielding the system
\begin{equation}
    \Lambda \partial_x z + M_3 z = {-} {\psi}.
\end{equation}

Let us look at the new operator:
\begin{equation}
    \mathcal{A}_{0} = -\Lambda \partial_{x} -M_{3} \text{Id},
\end{equation}
defined on the domain
    \begin{equation} 
    \begin{split}
    D(\mathcal{A}_{0}) &= \{z \in L^{2}, \; (h,v_{2},v)^{T} = Pz \;|\;v(0) = -b_{0}h(0),  v(L) = b_{1}h(L)+\mu c_{1}v_{2}(L),\; v_{2}(0) = 0\}.
    \end{split}
    \end{equation}

    Note that this can be expressed with the more explicit boundary conditions

\begin{equation}
\label{eq:bound1z}
    z_{3}(L) = D\begin{pmatrix}z_{1}(L) \\z_{2}(L)\end{pmatrix}, \qquad
    \begin{pmatrix}z_{1}(0) \\z_{2}(0)\end{pmatrix} = Ez_{3}(0).
\end{equation}
with
\begin{equation}
\label{eq:bound2z}
    D = \displaystyle \frac{\Delta(L)}{b_1 -\mu c_1 \Delta(L)} \begin{pmatrix}
        1 \\ - \mu c_1
    \end{pmatrix}^T, \; E = \begin{pmatrix} \displaystyle
        - \frac{b_0}{\Delta(0)} \\ -1
    \end{pmatrix}
\end{equation}
such that $D(\mathcal{A}_{0})$ can be expressed as
\begin{equation}
    D(\mathcal{A}_{0})=\{z\in L^{2}((0,L);\mathbb{R}^{3})\;|\; z \text{ satisfies }\eqref{eq:bound1z}\}.
\end{equation}
If we can show that there exists $\lambda_{0}>0$ such that $\mathcal{A}_{0}$ is surjective, then a fortiori there exists $z\in D(\mathcal{A}_{0})$ such that $\mathcal{A}_{0}z = {{\psi}}$ and therefore, using again the change of variable, there exists $(h,v_{2},v)$ such that \eqref{eq:wp2} holds and the proof of (iii) is done. To show this, note that this operator satisfies the assumption of \cite{lichtner2008spectral}, (with $F=G=0$ in the notations on \cite{lichtner2008spectral}, {recall that $V^{*}>0$ on $[0,L]$}).
Hence, from \cite{lichtner2008spectral} (Lemma 2.2, where our $\mathcal{A}_0$ is their $A$), $\mathcal{A}_{0}$ generates a $C^{0}$ semigroup on $L^{2}$ and therefore $(\lambda Id-\mathcal{A}_{0})^{-1}$ is well-defined on $\rho(\mathcal{A}_{0}) = \mathbb{C}\setminus\{\sigma_{p}(\mathcal{A}_{0})\}$, where $\sigma_{p}(\mathcal{A}_{0})$ is the point spectrum of $\mathcal{A}_{0}$, that is, the values $\lambda$ such that there exists $z\neq 0$ with
\begin{equation}
    \mathcal{A}_{0}z = \lambda z.
\end{equation}
If we can show that $0\in \rho(\mathcal{A}_{0})$, then $\mathcal{A}_{0}^{-1}$ is well defined on the image of $D(\mathcal{A})$ by $\mathcal{A}_{0}$, but since $\mathcal{A}_{0}$ is closed~\cite{lichtner2008spectral}, this is the whole $L^{2}$ and thus $A_{0}$ is surjective. It remains to show that there exists $\lambda_{0}>0$ such that 0 is not an eigenvalue of $\mathcal{A}_{0}$. Assume that 0 is an eigenvalue, using the change of variable again, there exists $u = (h,v_{2},v)^{T}\neq 0$ such that
\begin{equation}
\label{eq:contra1}
    M_{1} u_x + M_{2} u = 0,
\end{equation}
with $u$ satisfying the boundary conditions (ie. $u \in D(\mathcal{A}_0)$). Since $v_{2}=v_{x}$ from \eqref{eq:contra1}, it means that $(h,v)^{T}\neq 0$ is a solution to
\begin{equation}
\mathcal{A}(h,v)^{T} = \lambda_{0}(h,v)^{T},
\end{equation}
thus
\begin{equation}
\langle \mathcal{A}(h,v)^{T}, (h,v)^{T}\rangle= \lambda_{0}\|(h,v)^{T}\|^{2}_{V}>0,
\end{equation}
which is in contradiction with the dissipativity of the operator.
\section*{Acknowledgements}
The authors wish to thank the ANR-Tremplin StarPDE ANR-24-ERCS-0010. The authors also wish to thank Alexandre Ern for fruitful conversations.

\bibliographystyle{abbrv}
\bibliography{citations}

@book{bastin2016stability,
  title = {Stability and {{Boundary Stabilization}} of 1-{{D Hyperbolic Systems}}},
  author = {Bastin, Georges and Coron, Jean-Michel},
  year = {2016},
  series = {Progress in {{Nonlinear Differential Equations}} and {{Their Applications}}},
  volume = {88},
  publisher = {{Springer International Publishing}},
  address = {{Cham}},
  doi = {10.1007/978-3-319-32062-5},
  isbn = {978-3-319-32060-1 978-3-319-32062-5}
}

@article{hindmarsh1983odepack,
  title={ODEPACK, a systemized collection of ODE solvers},
  author={Hindmarsh, Alan C},
  journal={Scientific computing},
  year={1983},
  publisher={North-Holland}
}

@article{lichtner2008spectral,
  title = {Spectral Mapping Theorem for Linear Hyperbolic Systems},
  author = {Lichtner, Mark},
  year = {2008},
  month = feb,
  journal = {Proceedings of the American Mathematical Society},
  volume = {136},
  number = {6},
  pages = {2091--2101},
  issn = {0002-9939, 1088-6826},
  doi = {10.1090/S0002-9939-08-09181-8}
}

@article{lumer1961dissipative,
  title = {Dissipative Operators in a {{Banach}} Space},
  author = {Lumer, Gunter and Phillips, R. S.},
  year = {1961},
  month = jun,
  journal = {Pacific Journal of Mathematics},
  volume = {11},
  number = {2},
  pages = {679--698},
  issn = {0030-8730, 0030-8730},
  doi = {10.2140/pjm.1961.11.679}
}

@book{renardy2004introduction,
  title = {An Introduction to Partial Differential Equations},
  author = {Renardy, Michael and Rogers, Robert C.},
  year = {2004},
  series = {Texts in Applied Mathematics},
  edition = {2nd ed},
  number = {13},
  publisher = {{Springer}},
  address = {{New York}},
  isbn = {978-0-387-00444-0},
  lccn = {QA374 .R4244 2004}
}

@book{engel2000oneparameter,
  title = {One-{{Parameter Semigroups}} for {{Linear Evolution Equations}}},
  author = {Engel, Klaus-Jochen and Nagel, Rainer},
  year = {2000},
  series = {Graduate {{Texts}} in {{Mathematics}}},
  volume = {194},
  publisher = {{Springer-Verlag}},
  address = {{New York}},
  doi = {10.1007/b97696},
  isbn = {978-0-387-98463-6}
}

@article{bastin2011boundary,
  title = {On Boundary Feedback Stabilization of Non-Uniform Linear Hyperbolic Systems over a Bounded Interval},
  author = {Bastin, Georges and Coron, Jean-Michel},
  year = {2011},
  month = nov,
  journal = {Systems \& Control Letters},
  volume = {60},
  number = {11},
  pages = {900--906},
  issn = {01676911},
  doi = {10.1016/j.sysconle.2011.07.008}
}

@article{de1871theorie,
  title={Th{\'e}orie du mouvement non-permanent des eaux, avec application aux crues des rivi{\`e}res et {\`a} l’introduction des mar{\'e}es dans leur lit},
  author={Barr{\'e} de Saint-Venant, Adh{\'e}mar Jean-Claude},
  journal={Comptes Rendus de l'Acad{\'e}mie des Sciences},
  volume={73},
  number={147-154},
  pages={237--240},
  year={1871}
}

@INPROCEEDINGS{coron1999lyapunov,
  author={Coron, J. M. and d'Andréa-Novel, B. and Bastin, G.},
  booktitle={1999 European Control Conference (ECC)}, 
  title={A Lyapunov approach to control irrigation canals modeled by saint-venant equations}, 
  year={1999},
  volume={},
  number={},
  pages={3178-3183},
  doi={10.23919/ECC.1999.7099816}}

@ARTICLE{coron2007strict,
  author={Coron, Jean-Michel and d'Andrea-Novel, Brigitte and Bastin, Georges},
  journal={IEEE Transactions on Automatic Control}, 
  title={A Strict Lyapunov Function for Boundary Control of Hyperbolic Systems of Conservation Laws}, 
  year={2007},
  volume={52},
  number={1},
  pages={2-11},
  doi={10.1109/TAC.2006.887903}}

@article{coron2008dissipative,
  TITLE = {{Dissipative boundary conditions for one-dimensional nonlinear hyperbolic systems}},
  AUTHOR = {Coron, Jean-Michel and Bastin, Georges and d'Andr{\'e}a-Novel, Brigitte},
  URL = {https://minesparis-psl.hal.science/hal-00923596},
  JOURNAL = {{SIAM Journal on Control and Optimization}},
  PUBLISHER = {{Society for Industrial and Applied Mathematics}},
  VOLUME = {47},
  NUMBER = {3},
  PAGES = {1460-1498},
  YEAR = {2008},
  HAL_ID = {hal-00923596},
  HAL_VERSION = {v1},
}

@article{bastin2017quadratic,
title = {A quadratic Lyapunov function for hyperbolic density–velocity systems with nonuniform steady states},
journal = {Systems \& Control Letters},
volume = {104},
pages = {66-71},
year = {2017},
issn = {0167-6911},
doi = {https://doi.org/10.1016/j.sysconle.2017.03.013},
url = {https://www.sciencedirect.com/science/article/pii/S016769111730066X},
author = {Georges Bastin and Jean-Michel Coron}
}

@article{hayat2019quadratic,
  TITLE = {{A quadratic Lyapunov function for Saint-Venant equations with arbitrary friction and space-varying slope}},
  AUTHOR = {Hayat, Amaury and Shang, Peipei},
  URL = {https://hal.science/hal-01704710},
  JOURNAL = {{Automatica}},
  PUBLISHER = {{Elsevier}},
  VOLUME = {100},
  PAGES = {52--60},
  YEAR = {2019},
  HAL_ID = {hal-01704710},
  HAL_VERSION = {v1},
}

@article{hayat2021exponential,
  TITLE = {{Exponential stability of density-velocity systems with boundary conditions and source term for the $H^2$ norm}},
  AUTHOR = {Hayat, Amaury and Shang, Peipei},
  URL = {https://hal.science/hal-02190778},
  JOURNAL = {{Journal de Math{\'e}matiques Pures et Appliqu{\'e}es}},
  PUBLISHER = {{Elsevier}},
  VOLUME = {153},
  PAGES = {187-212},
  YEAR = {2021},
  DOI = {10.1016/j.matpur.2021.07.001},
  HAL_ID = {hal-02190778},
  HAL_VERSION = {v2},
}

@techreport{gerbeau2000derivation,
  TITLE = {{Derivation of Viscous Saint-Venant System for Laminar Shallow Water; Numerical Validation}},
  AUTHOR = {Gerbeau, Jean-Fr{\'e}d{\'e}ric and Perthame, Beno{\^i}t},
  URL = {https://inria.hal.science/inria-00072549},
  NOTE = {Projet M3N},
  TYPE = {Research Report},
  NUMBER = {RR-4084},
  INSTITUTION = {{INRIA}},
  YEAR = {2000},
  HAL_ID = {inria-00072549},
  HAL_VERSION = {v1},
}

@article{prieur2018boundary,
title = {Boundary feedback control of linear hyperbolic systems: Application to the Saint-Venant–Exner equations},
journal = {Automatica},
volume = {89},
pages = {44-51},
year = {2018},
issn = {0005-1098},
doi = {https://doi.org/10.1016/j.automatica.2017.11.028},
url = {https://www.sciencedirect.com/science/article/pii/S0005109817305708},
author = {Christophe Prieur and Joseph J. Winkin}
}

@misc{li2009exact,
      title={Exact boundary controllability and observability for first order quasilinear hyperbolic systems with a kind of nonlocal boundary conditions}, 
      author={Tatsien Li and Bopeng Rao and Zhiqiang Wang},
      year={2009},
      eprint={0908.1302},
      archivePrefix={arXiv},
      primaryClass={math.OC}
}

@article{halleux2003boundary, title={Boundary feedback control in networks of open channels}, author={Jonathan de Halleux and Christophe Prieur and Jean-Michel Coron and Brigitte d'Andr{\'e}a-Novel and Georges Bastin}, journal={Autom.}, year={2003}, volume={39}, pages={1365-1376}, url={https://api.semanticscholar.org/CorpusID:14252473} }

@inproceedings{gugat2009global,
  title={Global boundary controllability of the Saint-Venant system for sloped canals with friction},
  author={Gugat, Martin and Leugering, G{\"u}nter},
  booktitle={Annales de l'IHP Analyse non lin{\'e}aire},
  volume={26},
  number={1},
  pages={257--270},
  year={2009}
}

@article{gugat2003global,
title = {Global boundary controllability of the de St. Venant equations between steady states},
journal = {Annales de l'Institut Henri Poincaré C, Analyse non linéaire},
volume = {20},
number = {1},
pages = {1-11},
year = {2003},
issn = {0294-1449},
doi = {https://doi.org/10.1016/S0294-1449(02)00004-5},
url = {https://www.sciencedirect.com/science/article/pii/S0294144902000045},
author = {M. Gugat and G. Leugering}
}

@article{leugering2002modelling,
author = {Leugering, Guenter and Schmidt, J. P. Georg},
title = {On the Modelling and Stabilization of Flows in Networks of Open Canals},
journal = {SIAM Journal on Control and Optimization},
volume = {41},
number = {1},
pages = {164-180},
year = {2002},
doi = {10.1137/S0363012900375664},
URL = {    
        https://doi.org/10.1137/S0363012900375664
},
eprint = { 
        https://doi.org/10.1137/S0363012900375664
}
}

@article{gu2011exact,
author = {Gu, Qilong and Li, Tatsien},
title = {Exact Boundary Controllability for Quasilinear Hyperbolic Systems on a Tree-Like Network and Its Applications},
journal = {SIAM Journal on Control and Optimization},
volume = {49},
number = {4},
pages = {1404-1421},
year = {2011},
doi = {10.1137/080739902},
URL = {   
        https://doi.org/10.1137/080739902
},
eprint = { 
        https://doi.org/10.1137/080739902
}
}

@article{hayat2019boundary,
  title = {On Boundary Stability of Inhomogeneous 2 {\texttimes} 2 1-{{D}} Hyperbolic Systems for the {{{\emph{C}}}}{\textsuperscript{1}} Norm},
  author = {Hayat, Amaury},
  year = {2019},
  journal = {ESAIM: Control, Optimisation and Calculus of Variations},
  volume = {25},
  pages = {82},
  publisher = {EDP Sciences},
  issn = {1292-8119, 1262-3377},
  doi = {10.1051/cocv/2018059},
  copyright = {https://www.edpsciences.org/en/authors/copyright-and-licensing}
}

@article{hayat2023pi,
  title={PI control for the cascade channels modeled by general Saint-Venant equations},
  author={Hayat, Amaury and Hu, Yating and Shang, Peipei},
  journal={IEEE Transactions on Automatic Control},
  year={2023},
  publisher={IEEE}
}

@inproceedings{mascia2005asymptotic,
  TITLE = {{Asymptotic stability of steady-states for Saint-Venant equations with real viscosity}},
  AUTHOR = {Mascia, Corrado and Rousset, Fr{\'e}d{\'e}ric},
  URL = {https://hal.science/hal-00458162},
  BOOKTITLE = {{Analysis and simulation of fluid dynamics}},
  ADDRESS = {Lille, France},
  PUBLISHER = {{Birkh{\"a}user}},
  SERIES = {Advances in mathematical fluid mechanics},
  PAGES = {155-162},
  YEAR = {2005},
  HAL_ID = {hal-00458162},
  HAL_VERSION = {v1},
}

@misc{hayat2018exponentialstabilitygeneral1d,
      title={Exponential stability of general 1-D quasilinear systems with source terms for the $C^1$ norm under boundary conditions}, 
      author={Amaury Hayat},
      year={2018},
      eprint={1801.02353},
      archivePrefix={arXiv},
      primaryClass={math.AP},
      url={https://arxiv.org/abs/1801.02353}, 
}

@article{dick2011strict,
  title={A strict $H^1$-Lyapunov function and feedback stabilization for the isothermal Euler equations with friction},
  author={Dick, Markus and Gugat, Martin and Leugering, G{\"u}nter},
  journal={Numerical Algebra, Control and Optimization},
  volume={1},
  number={2},
  pages={225--244},
  year={2011},
  publisher={Numerical Algebra, Control and Optimization}
}

@article{GREENBERG198466,
title = {The effect of boundary damping for the quasilinear wave equation},
journal = {Journal of Differential Equations},
volume = {52},
number = {1},
pages = {66-75},
year = {1984},
issn = {0022-0396},
doi = {https://doi.org/10.1016/0022-0396(84)90135-9},
url = {https://www.sciencedirect.com/science/article/pii/0022039684901359},
author = {Greenberg, J.M and Li, Ta Tsien}
}

@article {Qin,
    AUTHOR = {Qin, Tie Hu},
     TITLE = {Global smooth solutions of dissipative boundary value problems
              for first order quasilinear hyperbolic systems},
      NOTE = {A Chinese summary appears in Chinese Ann. Math. Ser. A
              {{\bf{6}}} (1985), no. 4, 514},
   JOURNAL = {Chinese Ann. Math. Ser. B},
  FJOURNAL = {Chinese Annals of Mathematics. Series B. Shuxue Niankan. B Ji},
    VOLUME = {6},
      YEAR = {1985},
    NUMBER = {3},
     PAGES = {289--298},
      ISSN = {0252-9599},
   MRCLASS = {35L60},
  MRNUMBER = {842971},
MRREVIEWER = {R. G. Airapetyan},
}

@book{doi:10.1137/1.9780898718607,
author = {Krstic, Miroslav and Smyshlyaev, Andrey},
title = {Boundary Control of PDEs},
publisher = {Society for Industrial and Applied Mathematics},
year = {2008},
doi = {10.1137/1.9780898718607},
address = {Philadelphia, PA},
edition   = {},
URL = {https://epubs.siam.org/doi/abs/10.1137/1.9780898718607},
eprint = {https://epubs.siam.org/doi/pdf/10.1137/1.9780898718607}
}

@article{KRSTIC2008750,
title = {Backstepping boundary control for first-order hyperbolic PDEs and application to systems with actuator and sensor delays},
journal = {Systems \& Control Letters},
volume = {57},
number = {9},
pages = {750-758},
year = {2008},
issn = {0167-6911},
doi = {https://doi.org/10.1016/j.sysconle.2008.02.005},
url = {https://www.sciencedirect.com/science/article/pii/S0167691108000352},
author = {Miroslav Krstic and Andrey Smyshlyaev}
}

@INPROCEEDINGS{6160338,
  author={Vazquez, Rafael and Krstic, Miroslav and Coron, Jean-Michel},
  booktitle={2011 50th IEEE Conference on Decision and Control and European Control Conference}, 
  title={Backstepping boundary stabilization and state estimation of a 2 × 2 linear hyperbolic system}, 
  year={2011},
  volume={},
  number={},
  pages={4937-4942},
  doi={10.1109/CDC.2011.6160338}}

@article{Meglio2013StabilizationOA,
  title={Stabilization of a System of  \$n+1\$ Coupled First-Order Hyperbolic Linear PDEs With a Single Boundary Input},
  author={Florent Di Meglio and Rafael V{\'a}zquez and Miroslav Krsti{\'c}},
  journal={IEEE Transactions on Automatic Control},
  year={2013},
  volume={58},
  pages={3097-3111},
  url={https://api.semanticscholar.org/CorpusID:13292262}
}

@INPROCEEDINGS{7402381,
  author={Diagne, Ababacar and Diagne, Mamadou and Tang, Shuxia and Krstic, Miroslav},
  booktitle={2015 54th IEEE Conference on Decision and Control (CDC)}, 
  title={Backstepping stabilization of the linearized Saint-Venant-Exner Model: Part I - state feedback}, 
  year={2015},
  volume={},
  number={},
  pages={1242-1247},
  doi={10.1109/CDC.2015.7402381}}

@article{DIAGNE2016130,
title = {State Feedback Stabilization of the Linearized Bilayer Saint-Venant Model},
journal = {IFAC-PapersOnLine},
volume = {49},
number = {8},
pages = {130-135},
year = {2016},
note = {2nd IFAC Workshop on Control of Systems Governed by Partial Differential Equations CPDE 2016},
issn = {2405-8963},
doi = {https://doi.org/10.1016/j.ifacol.2016.07.431},
url = {https://www.sciencedirect.com/science/article/pii/S2405896316306395},
author = {Ababacar Diagne and Shuxia Tang and Mamadou Diagne and Miroslav Krstic}
}

@inproceedings{dimeglio:hal-01499940,
  TITLE = {{A backstepping boundary observer for a class of linear first-order hyperbolic systems}},
  AUTHOR = {Di Meglio, Florent and Krstic, Miroslav and Vazquez, Rafael},
  URL = {https://minesparis-psl.hal.science/hal-01499940},
  BOOKTITLE = {{2013 European Control Conference (ECC 2013)}},
  ADDRESS = {Z{\"u}rich, Switzerland},
  SERIES = {Proceedings of the 2013 European Control Conference (ECC 2013)},
  PAGES = {1597 - 1602},
  YEAR = {2013},
  MONTH = Jul,
  HAL_ID = {hal-01499940},
  HAL_VERSION = {v1},
}

@article{gugat2012h,
  title={H$^2$-Stabilization of the Isothermal Euler Equations: a Lyapunov function approach},
  author={Gugat, Martin and Leugering, G{\"u}nter and Tamasoiu, Simona and Wang, Ke},
  journal={Chinese Annals of Mathematics, Series B},
  volume={33},
  number={4},
  pages={479--500},
  year={2012},
  publisher={Springer}
}

@article{bastin2023diffusion,
  title={Diffusion and robustness of boundary feedback stabilization of hyperbolic systems},
  author={Bastin, Georges and Coron, Jean-Michel and Hayat, Amaury},
  journal={Mathematics of Control, Signals, and Systems},
  volume={35},
  number={1},
  pages={159--185},
  year={2023},
  publisher={Springer}
}

@article{bastin2025usefulness,
  title={The usefulness of viscosity for the robustness of boundary feedback control of an unstable fluid flow system},
  author={Bastin, Georges and Coron, Jean-Michel and Hayat, Amaury},
  journal={Automatica},
  volume={173},
  pages={112048},
  year={2025},
  publisher={Elsevier}
}

@article{karafyllis2022spill,
  title={Spill-free transfer and stabilization of viscous liquid},
  author={Karafyllis, Iasson and Krstic, Miroslav},
  journal={IEEE Transactions on Automatic Control},
  volume={67},
  number={9},
  pages={4585--4597},
  year={2022},
  publisher={IEEE}
}

@article{karafyllis2022feedback,
  title={Feedback stabilization of tank-liquid system with robustness to wall friction},
  author={Karafyllis, Iasson and Vokos, Filippos and Krstic, Miroslav},
  journal={ESAIM: Control, Optimisation and Calculus of Variations},
  volume={28},
  pages={81},
  year={2022},
  publisher={EDP Sciences}
}

@article{arfaoui2011boundary,
  title={Boundary stabilizability of the linearized viscous Saint-Venant system},
  author={Arfaoui, Hassen and Belgacem, F Ben and El Fekih, Henda and Raymond, Jean-Pierre},
  journal={Discrete Contin. Dyn. Syst. Ser. B},
  volume={15},
  number={3},
  pages={491--511},
  year={2011}
}

\end{document}